\documentclass[11pt]{article}
\usepackage[T1]{fontenc}
\usepackage{amsmath}
\usepackage{amsfonts}
\usepackage{amsthm}
\usepackage{amssymb}
\usepackage{makeidx}
\usepackage{graphicx}
\usepackage{lmodern}
\usepackage[left=2.2cm,right=2.2cm,top=2cm,bottom=2cm]{geometry}

\usepackage{color}
\usepackage{comment}
\usepackage[dvipsnames]{xcolor}
\usepackage{graphicx}
\usepackage{framed}
\usepackage{tikz}
\usepackage{calrsfs}
\DeclareMathAlphabet{\pazocal}{OMS}{zplm}{m}{n}

\usepackage{fourier-orns}
\usepackage{mathtools}
\usepackage{relsize}
\usepackage{dsfont}

\newcommand{\R}{\mathbb{R}}

\newcommand{\J}{\mathbb{J}}

\newcommand{\F}{\pazocal{F}}
\newcommand{\G}{\pazocal{G}}
\newcommand{\U}{\pazocal{U}}
\newcommand{\K}{\pazocal{K}}
\newcommand{\V}{\pazocal{V}}
\newcommand{\M}{\pazocal{M}}

\newcommand{\T}{\mathcal{T}}

\newcommand{\Acal}{\mathcal{A}}
\newcommand{\Mcal}{\mathcal{M}}

\newcommand{\Lcal}{\mathcal{L}}

\newcommand{\Pcal}{\mathcal{P}}
\newcommand{\Ccal}{\mathcal{C}}
\newcommand{\Scal}{\mathcal{S}}
\newcommand{\Hcal}{\mathcal{H}}
\newcommand{\Ncal}{\mathcal{N}}

\newcommand{\Spazo}{\pazocal{S}}
\newcommand{\Hpazo}{\pazocal{H}}
\newcommand{\Ppazo}{\pazocal{P}}
\newcommand{\Cpazo}{\pazocal{C}}

\newcommand{\D}{\textnormal{D}}

\newcommand{\supp}{\textnormal{supp}}
\newcommand{\Lip}{\textnormal{Lip}}
\newcommand{\MP}{\textnormal{MP}}
\newcommand{\Id}{\textnormal{Id}}
\newcommand{\loc}{\textnormal{loc}}

\newcommand{\INTDom}[3]{\int_{#2} #1 \textnormal{d} #3}
\newcommand{\INTSeg}[4]{\int_{#3}^{#4} #1 \textnormal{d} #2}
\newcommand{\NormL}[3]{\parallel \hspace{-0.1cm} #1 \hspace{-0.1cm} \parallel _ {L^{#2}(#3)}}
\newcommand{\NormC}[3]{\left\| #1  \right\| _ {C^{#2}(#3)}}
\newcommand{\Norm}[1]{\parallel \hspace{-0.1cm} #1 \hspace{-0.1cm} \parallel}
\newcommand{\Bgamma}{\boldsymbol{\gamma}}
\newcommand{\Bpartial}{\boldsymbol{\partial}}

\newcommand{\derv}[3]{\frac{\textnormal{d}^{#3} #1}{\textnormal{d} #2^{#3}}}

\newcommand{\Bmu}{\mbox{$\raisebox{-0.59ex}
  {$l$}\hspace{-0.18em}\mu\hspace{-0.88em}\raisebox{-0.98ex}{\scalebox{2}
  {$\color{white}.$}}\hspace{-0.416em}\raisebox{+0.88ex}
  {$\color{white}.$}\hspace{0.46em}$}{}}
  
\newcommand{\BGamma}{\mbox{$ \textnormal{l} \hspace{-0.1em}\Gamma\hspace{-0.88em}\raisebox{-0.98ex}{\scalebox{2}
  {$\color{white}.$}}\hspace{0.46em}$}{}}

\newtheorem{rmk}{Remark}
\newtheorem{lem}{Lemma}
\newtheorem{Def}{Definition}

\newtheorem{thm}{Theorem}
\newtheorem{prop}{Proposition}

\numberwithin{equation}{section}
\numberwithin{lem}{section}
\numberwithin{thm}{section}
\numberwithin{prop}{section}
\numberwithin{Def}{section}

\renewenvironment{framed}[1][\hsize]
   {\MakeFramed{\hsize#1\advance\hsize-\width \FrameRestore}}%
   {\endMakeFramed}

\numberwithin{equation}{section}
\numberwithin{lem}{section}
\numberwithin{thm}{section}
\numberwithin{prop}{section}
\numberwithin{Def}{section}

\title{A Pontryagin Maximum Principle in Wasserstein Spaces for Constrained Optimal Control Problems}
\author{Beno\^it Bonnet\footnote{Aix Marseille Universit\'e, CNRS, ENSAM, Universit\'e de Toulon, LIS, Marseille, France. \textit{benoit.bonnet@lis-lab.fr}}} 

\date{\today}

\begin{document}

\maketitle
\begin{abstract} 
In this paper, we prove a Pontryagin Maximum Principle for constrained optimal control problems in the Wasserstein space of probability measures. The dynamics is described by a transport equation with non-local velocities which are affine in the control, and is subject to end-point and running state constraints. Building on our previous work, we combine the classical method of needle-variations from geometric control theory and the metric differential structure of the Wasserstein spaces to obtain a maximum principle formulated in the so-called Gamkrelidze form.
\end{abstract}

\section{Introduction}

Transport equations with non-local velocities have drawn a great amount of attention from several scientific communities for almost a century. They were first introduced in statistical physics to describe averaged Coulomb interactions within large assemblies of particles (see e.g. \cite{vlasov}), and are still to this day a widely studied topic in mathematical physics. More recently, a growing interest in the mathematical modelling of multi-agent systems has brought to light a whole new panel of problems in which these equations play a central role. Starting from the seminal paper of Cucker and Smale \cite{CS2} dealing with emergent behaviour in animal flocks, a large literature has been devoted to the fine mathematical analysis of kinetic cooperative systems, i.e. systems described by non-local dynamics with attractive velocities, see e.g. \cite{HK,lifebelt,bellomotosin,albicristiani}. Besides, several prominent papers aimed at describing the emergence of patterns which were initially discovered for systems of ODEs in the context of kinetic models described by continuity equations \cite{Carrillo2010,HaLiu}. Simultaneously, Lasry and Lions laid in \cite{Lasry2007} the foundations of the theory of mean-field games, which is today one of the most active communities working on variational problems involving continuity equations, see e.g. \cite{Cardaliaguet2012,CDLL} and references therein.

Later on, the focus shifted partly to include control-theoretic problems such as reachability analysis, optimal control, or explicit design of sparse control strategies. For these purposes, the vast majority of the existing contributions have taken advantage of the recent developments in the theory of optimal transport. We refer the reader to \cite{villani1,OTAM} for a comprehensive introduction to this ever-expanding topic. In particular, the emergence of powerful tools of analysis in the so-called Wasserstein spaces has allowed for the establishment of a general existence theory for non-local transport equations (see e.g. \cite{AGS,AmbrosioGangbo}), which incorporates natural Lipschitz and metric estimates in the smooth cases (see \cite{Pedestrian}). 

Apart from a few controllability results as in \cite{michel}, most of the attention of the community has been devoted to optimal control problems in Wasserstein spaces. The existence of optimal solutions has been investigated with various degrees of generality in \cite{achdou2,achdou1,FPR,MFOC,FLOS,Pogodaev2016}, mostly by means of $\Gamma$-convergence arguments. Besides, a few papers have been dealing with numerical methods either in the presence of diffusion terms, which considerably simplify the convergence analysis of the corresponding schemes (see e.g. \cite{RobotSwarms}), or in the purely metric setting \cite{AlbiPareschiZanella,Burger2019,Pogodaev2017}.

The derivation of Hamilton-Jacobi and Pontryagin optimality conditions has been an active topic in the community of Wasserstein optimal control in the recent years. Starting from the seminal paper \cite{HJBWasserstein} on Hamilton-Jacobi equations in the Wasserstein space, several  contributions such as \cite{CavagnariMP,Cavagnari2018} have been aiming at refining a dynamic-programming principle for mean-field optimal control problems. Pontryagin-type optimality conditions, on the other hand, have received less interest. The first result derived in \cite{MFPMP} focuses on a multi-scale ODE-PDE system in which the control only acts on the ODE part. In this setting, the Pontryagin Maximum Principle (``PMP'' for short) is derived by combining $\Gamma$-convergence and mean-field limit arguments. Another approach, introduced in our previous work \cite{PMPWass}, studies the infinite-dimensional problem by means of the classical technique of needle-variations (see e.g. \cite{AgrachevSachkov,BressanPiccoli}) and makes an extensive use of the theory of Wasserstein subdifferential calculus formalized in \cite{AGS}. The corresponding maximum principle is formulated as a Hamiltonian flow in the space of measures in the spirit of \cite{AmbrosioGangbo}, and is in a sense the most natural generalization to be expected of the usual finite-dimensional Pontryagin-type optimality conditions. We would also like to mention that in \cite{Burger2019,Pogodaev2016}, the authors derived first-order necessary optimality conditions for special classes of optimal control problems on continuity equations, via methods which are quite distinct from that which we already sketched.

It is worth noticing that optimal control problems in Wasserstein spaces bear a lot of similarities with mean-field games. It was highlighted as early as \cite{Lasry2007} and further detailed e.g. in \cite{Carmona2013} that in the class of so-called potential mean-field games, the self-organization of an ensemble of agents could be equivalently reformulated as an optimal control problem in Wasserstein spaces involving adequately modified functionals. In particular in \cite{Carmona2015}, a PMP was derived for controlled McKean-Vlasov dynamics describing such an optimal control problem from a probabilistic point of view. The analysis therein is carried out by leveraging the formalism of Lions derivatives in Wasserstein spaces (see e.g. \cite{Cardaliaguet2012}), which are one of the possible equivalent ways of looking at derivatives in the metric space of probability measures.

In this paper, we further the line of research initiated in \cite{PMPWass} by extending our previous result, consisting in a Pontryagin Maximum Principle in Wasserstein spaces, to the setting of constrained optimal control problems. Such problems can be written in the following general form 
\begin{equation*}
(\Ppazo) ~~ \left\{
\begin{aligned}
\underset{u \in \U}{\text{min}} & \left[ \INTSeg{L(t,\mu(t),u(t))}{t}{0}{T} + \varphi(\mu(T)) \right], \\
\text{s.t.} 
& \left\{ 
\begin{aligned}
& \partial_t \mu(t) + \nabla \cdot ((v[\mu(t)](t,\cdot) + u(t,\cdot)) \mu(t)) = 0, \\
& \mu(0) = \mu^0 \in \Pcal_c(\R^d),
\end{aligned}
\right. \\
\text{and}
& \left\{
\begin{aligned}
\Psi^I(\mu(T)) & \leq 0,~ \Psi^E(\mu(T)) = 0 ,  \\
\Lambda(t,\mu(t)) & \leq 0 ~~ \text{for all $t \in [0,T]$.}
\end{aligned}
\right.
\end{aligned}
\right.
\end{equation*}
Here, the functions $(t,\mu,\omega) \mapsto L(t,\mu,\omega)$ and $\mu \mapsto \varphi(\mu)$ describe running and final costs, while the maps $(t,\mu) \mapsto \Lambda(t,\mu)$ and $\mu \mapsto \Psi^I(\mu),\Psi^E(\mu)$ are running and end-point constraints respectively. The velocity field $(t,x,\mu) \mapsto v[\mu](t,x)$ is a general non-local drift, which can be given e.g. in the form of a convolution (see \cite{CS2,ControlKCS,FPR}). The control $(t,x) \mapsto u(t,x)$ is a vector-field which depends on both time and space, as customary in distributed control of partial differential equations (see e.g. \cite{Troltzsch}).

The methodology that we follow relies on the technique of \textit{packages of needle-variations}, combined with a \textit{Lagrange multiplier rule}. In essence, this method allows to recover the maximum principle from a family of finite-dimensional first-order optimality conditions by means of the introduction of a suitable costate. Even though classical in the unconstrained case, this direct approach does require some care to be translated to constrained problem. Indeed, the presence of constraints induces an unwanted dependency between the Lagrange multipliers and the needle-parameters. This extra difficulty can be circumvented by considering $N$-dimensional perturbations of the optimal trajectory instead of a single one, and by performing a limiting procedure as $N$ goes to infinity. Originally introduced in \cite{Arutyunov2004} for smooth optimal control problems with end-point constraints, this approach was extended in \cite{NonsmoothPMPNeedle2} to the case of non-smooth and state-constrained problems. When trying to further adapt this method to the setting of Wasserstein spaces, one is faced with an extra structural difficulty. In the classical statement of the maximum principle, the presence of state constraints implies a mere $BV$ regularity in time for the covectors. However, a deep result of optimal transport theory states that solutions of continuity equations in Wasserstein spaces coincide exactly with absolutely continuous curves (see e.g. \cite[Theorem 8.3.1]{AGS}). Whence, in order to write a well-defined Wasserstein Hamiltonian flow in the spirit of \cite{AmbrosioGangbo,PMPWass}, we choose to formulate a maximum principle in the so-called Gamkrelidze form (see e.g. \cite{Arutyunov2011}), which allows to recover a stronger absolute continuity in time of the costates at the price of an extra regularity assumption on the state constraints.
 
This article is structured as follows. In Section \ref{section:Recall}, we recall general results of analysis in measure spaces along with elements of subdifferential calculus in Wasserstein spaces and existence results for continuity equations. We also introduce several notions of non-smooth analysis, including a general Lagrange multiplier rule formulated in terms of \textit{Michel-Penot} subdifferentials. In Section \ref{section:PMP_General}, we state and prove our main result, that is Theorem \ref{thm:PMP_General}. The argument is split into four steps which loosely follow the methodology already introduced in \cite{PMPWass}. We exhibit in Appendix \ref{appendix:Examples} a series of examples of functionals satisfying the structural assumptions \textbf{(H)} of Theorem \ref{thm:PMP_General}, and we provide in Appendix \ref{appendix:ConstraintPenalization} the analytical expression of the Wasserstein gradient of a functional involved in the statement of Theorem \ref{thm:PMP_General}.


\section{Preliminary results} 
\label{section:Recall}

In this section, we recall several notions about analysis in the space of measures,  optimal transport theory, Wasserstein spaces, subdifferential calculus in the space $(\Pcal_2(\R^d),W_2)$,  and continuity equations with non-local velocities.  We also introduce some elementary notions of non-smooth calculus in Banach spaces. For a complete introduction to these topics, see \cite{AGS,villani1,OTAM} and \cite{MichelPenot,Clarke} respectively. 


\subsection{Analysis in measure spaces and the optimal transport problem}

In this section, we introduce some classical notations and results of measure theory, optimal transport and analysis in Wasserstein spaces. We denote by $(\M_+(\R^d),\Norm{\cdot}_{TV})$ the set of real-valued non-negative Borel measures defined over $\R^d$ endowed with the total variation norm, and by $\Lcal^d$ the standard Lebesgue measure on $\R^d$. It is known by Riesz Theorem (see e.g. \cite[Theorem 1.54]{AmbrosioFuscoPallara}) that this space can be identified with the topological dual of the Banach space $(C^0_0(\R^d),\Norm{\cdot}_{C^0})$, which is the completion of the space of continuous and compactly supported functions $C^0_c(\R^d)$ endowed with the $C^0$-norm.

We denote by $\Pcal(\R^d) \subset \M_+(\R^d)$ the set of Borel probability measures, and for $p \geq 1$, we denote by $\Pcal_p(\R^d) \subset\Pcal(\R^d)$ the set of measures with finite $p$-th moment, i.e.
\begin{equation*}
\Pcal_p(\R^d) = \left\{ \mu \in \Pcal(\R^d) ~\text{s.t.}~ \INTDom{|x|^p}{\R^d}{\mu(x)} < +\infty \right\}.
\end{equation*}
The \textit{support} of a Borel measure $\mu \in \M_+(\R^d)$ is defined as the closed set $\supp(\mu) = \{ x \in \R^d ~\text{s.t.}~ \mu(\pazocal{N}) > 0 ~\text{for any neighbourhood $\pazocal{N}$ of $x$}\}$. We denote by $\Pcal_c(\R^d) \subset \Pcal(\R^d)$ the set of probability measures with compact support. 

We say that a sequence $(\mu_n) \subset \Pcal(\R^d)$ of Borel probability measures \textit{converges narrowly} towards $\mu \in \Pcal(\R^d)$ -- denoted by $\mu_n \underset{n \rightarrow +\infty}{\rightharpoonup^*} \mu$ -- provided that
\begin{equation}
\label{eq:Narrow_convergence}
\INTDom{\phi(x)}{\R^d}{\mu_n(x)} ~\underset{n \rightarrow + \infty}{\longrightarrow}~ \INTDom{\phi(x)}{\R^d}{\mu(x)}, 
\end{equation}
for all $\phi \in C^0_b(\R^d)$. Here $C^0_b(\R^d)$ denotes the set of continuous and bounded functions over $\R^d$. Remark that the corresponding \textit{narrow topology} coincides with the restriction to $\Pcal(\R^d)$ of the weak-$^*$ topology in $\M_+(\R^d)$. In the following proposition, we recall a useful convergence property on the measures of Borel sets, usually referred to as the \textit{Portmanteau Theorem} (see e.g. \cite[Proposition 1.62]{AmbrosioFuscoPallara}). 

\begin{prop}[Portmanteau Theorem]
\label{prop:Portmanteau}
Let $(\mu_n) \subset \M_+(\R^d)$ be a sequence of measures converging in the weak-$^*$ topology towards $\mu \in \M_+(\R^d)$. Then for any Borel set $A \subset \R^d$ such that $\mu(\partial A) = 0$, it holds that $\mu_n(A) \underset{n \rightarrow +\infty}{\longrightarrow} \mu(A).$
\end{prop} 

We recall in the following definitions the notion of \textit{pushforward} of a Borel probability measure through a Borel map, along with that of \textit{transport plan}. 

\begin{Def}[Pushforward of a measure through a Borel map] 
Given $\mu \in \Pcal(\R^d)$ and a Borel map $f : \R^d \rightarrow \R^d$, the \textnormal{pushforward} $f_{\#} \mu$ of $\mu$ through $f(\cdot)$ is the unique Borel probability measure which satisfies $f_{\#} \mu (B) = \mu(f^{-1}(B))$ for any Borel set $B \subset \R^d$. 
\end{Def}

\begin{Def}[Transport plan]
Let $\mu,\nu \in \Pcal(\R^d)$. We say that $\gamma \in \Pcal(\R^{2d})$ is a \textnormal{transport plan} between $\mu$ and $\nu$ -- denoted by $\gamma \in \Gamma(\mu,\nu)$ -- provided that 
\begin{equation*}
\gamma(A \times \R^d) = \mu(A) \qquad \text{and} \qquad \gamma(\R^d \times B) = \nu(B), 
\end{equation*}
for any pair of Borel sets $A,B \subset \R^d$. This property can be equivalently formulated in terms of pushforwards by $\pi^1_{\#} \gamma = \mu$ and $\pi^2_{\#} \gamma = \nu$, where the maps $\pi^1,\pi^2 : \R^{2d} \rightarrow \R^d$ denote the projection operators on the first and second factor respectively.
\end{Def}

In 1942, the Russian mathematician Leonid Kantorovich introduced in \cite{Kantorovich1942} the \textit{optimal mass transportation problem} in its modern mathematical formulation. Given two probability measures $\mu,\nu \in \Pcal(\R^d)$ and a cost function $c : \R^{2d} \rightarrow (-\infty,+\infty]$, one aims at finding a \textit{transport plan} $\gamma \in \Gamma(\mu,\nu)$ such that 
\begin{equation*}
\INTDom{c(x,y)}{\R^{2d}}{\gamma(x,y)} = \min_{\gamma'} \left\{ \INTDom{c(x,y)}{\R^{2d}}{\gamma'(x,y)} ~~ \text{s.t.} ~ \gamma' \in \Gamma(\mu,\nu) \right\}.
\end{equation*}
This problem has been extensively studied in very broad contexts (see e.g. \cite{AGS,villani1}) with high levels of generality on the underlying spaces and cost functions.  In the particular case where $c(x,y) = |x-y|^p$ for some real number $p \geq 1$, the optimal transport problem can be used to define a distance over the subset $\Pcal_p(\R^d)$ of $\Pcal(\R^d)$.

\begin{Def}[Wasserstein distance and Wasserstein spaces] 
Given two probability measures $\mu,\nu \in \Pcal_p(\R^d)$, the \textnormal{$p$-Wasserstein distance} between $\mu$ and $\nu$ is defined by
\begin{equation*}
W_p(\mu,\nu) = \min_{\gamma} \left\{ \left( \INTDom{|x-y|^p}{\R^{2d}}{\gamma(x,y)} \right)^{1/p} ~~ \text{s.t.} ~ \gamma \in \Gamma(\mu,\nu) \right\}.
\end{equation*}
The set of plans $\gamma \in \Gamma(\mu,\nu)$ achieving this optimal value is denoted by $\Gamma_o(\mu,\nu)$ and is referred to as the set of \textnormal{optimal transport plans} between $\mu$ and $\nu$. The space $(\Pcal_p(\R^d),W_p)$ of probability measures with finite $p$-th moment endowed with the $p$-Wasserstein metric is called the \textnormal{Wasserstein space} of order $p$.
\end{Def}

We recall some of the interesting properties of these spaces in the following proposition (see e.g. \cite[Chapter 7]{AGS} or \cite[Chapter 6]{villani1}).

\begin{prop}[Properties of the Wasserstein spaces]
\label{prop:Properties_Wp}
The Wasserstein spaces $(\Pcal_p(\R^d),W_p)$ are separable geodesic spaces. The topology generated by the $p$-Wasserstein metric metrizes the weak-$^*$ topology of probability measures induced by the narrow convergence \eqref{eq:Narrow_convergence}. More precisely, it holds that
\begin{equation*}
W_p(\mu_n,\mu) \underset{n \rightarrow +\infty}{\longrightarrow} 0 \qquad \text{if and only if} \qquad
\left\{
\begin{aligned}
\mu_n & \underset{n \rightarrow +\infty}{\rightharpoonup^*} \mu, \\
\INTDom{|x|^p}{\R^d}{\mu_n(x)} & \underset{n \rightarrow +\infty}{\longrightarrow} \INTDom{|x|^p}{\R^d}{\mu(x)}.
\end{aligned}
\right.
\end{equation*}
Given two measures $\mu,\nu \in \Pcal(\R^d)$, the Wasserstein distances are ordered, i.e. $W_{p_1}(\mu,\nu) \leq W_{p_2}(\mu,\nu)$ whenever $p_1 \leq p_2$. Moreover, when $p = 1$, the following \textnormal{Kantorovich-Rubinstein duality formula} holds 
\begin{equation} \label{eq:Kantorovich_duality}
W_1(\mu,\nu) = \sup_{\phi} \left\{ \INTDom{\phi(x) \,}{\R^d}{(\mu-\nu)(x)} ~\text{s.t.}~ \Lip(\phi;\R^d) \leq 1 ~\right\},
\end{equation}
where $\Lip(\phi;\R^d)$ denotes the Lipschitz constant of the map $\phi(\cdot)$ on $\R^d$.
\end{prop}

In what follows, we shall mainly restrict our considerations to the Wasserstein spaces of order 1 and 2 built over $\Pcal_c(\R^d)$. We end these introductory paragraphs by recalling the concepts of \textit{disintegration} and \textit{barycenter} in the context of optimal transport.

\begin{Def}[Disintegration and barycenter]
\label{def:Barycenter}
Let $\mu,\nu \in \Pcal_p(\R^d)$ and $\gamma \in \Gamma(\mu,\nu)$ be a transport plan between $\mu$ and $\nu$. We define the \textnormal{disintegration} $\{ \gamma_x\}_{x \in \R^d} \subset \Pcal_p(\R^d)$ of $\gamma$ on its first marginal $\mu$, usually denoted by $\gamma = \int \gamma_x \textnormal{d} \mu(x)$, as the $\mu$-almost uniquely determined Borel family of probability measures such that 
\begin{equation*}
\INTDom{\phi(x,y)}{\R^{2d}}{\gamma(x,y)} = \INTDom{\INTDom{\phi(x,y)}{\R^d}{\gamma_x(y)}}{\R^d}{\mu(x)}, 
\end{equation*}
for any Borel map $\phi \in L^1(\R^{2d},\R;\gamma)$. The \textnormal{barycenter} $\bar{\gamma} \in L^p(\R^d,\R^d;\mu)$ of the plan $\gamma$ is then defined by 
\begin{equation*}
\bar{\gamma}(x) = \INTDom{y \,}{\R^d}{\gamma_x(y)}.
\end{equation*}
for $\mu$-almost every $x \in \R^d$. 
\end{Def}

\begin{prop}[Wasserstein estimate between disintegrations]
\label{prop:Estimation_Barycenter}
Let $\mu \in \Pcal_c(\R^d)$, $\gamma^1 = \int \gamma^1_x \textnormal{d} \mu(x) \in \Pcal_c(\R^{2d})$ and $\gamma^2 = \int \gamma^2_x \textnormal{d} \mu(x) \in \Pcal_c(\R^{2d})$. Then, it holds that
\begin{equation}
\label{eq:Estimation_Barycenter}
W_1(\gamma^1,\gamma^2) \leq \INTDom{W_1(\gamma^1_x,\gamma^2_x)}{\R^d}{\mu(x)}.
\end{equation}
\end{prop}
\begin{proof}
Take $\xi \in \Lip(\R^{2d})$ with  $\Lip(\xi;\R^{2d}) \leq 1$. One has that 
\begin{equation*}
\INTDom{\xi(x,r)}{\R^{2d}}{(\gamma^1-\gamma^2)(x,r)} = \INTDom{\INTDom{\xi(x,r)}{\R^d}{(\gamma^1_x - \gamma^2_x)(r)}}{\R^d}{\mu(x)} \leq \INTDom{W_1(\gamma^1_x,\gamma^2_x)}{\R^d}{\mu(x)}
\end{equation*}
by Kantorovich duality \eqref{eq:Kantorovich_duality} since the maps $r \mapsto \xi(x,r)$ are 1-Lipschitz for all $x \in \R^d$. Taking now the supremum over $\xi \in \Lip(\R^{2d})$ with $\Lip(\xi;\R^{2d}) \leq 1$  yields the desired estimate, again as a consequence of \eqref{eq:Kantorovich_duality}. 
\end{proof}


\subsection{Subdifferential calculus in $(\Pcal_2(\R^d),W_2)$}

In this section, we recall some key notions of subdifferential calculus in Wasserstein spaces. We also prove in Proposition \ref{prop:Chainrule} a general chainrule formula along multi-dimensional families of perturbations for sufficiently regular functional defined over $\Pcal_c(\R^d)$. We refer the reader to \cite[Chapters 9-11]{AGS} for a thorough introduction to the theory of subdifferential calculus in Wasserstein spaces, as well as to \cite{WassDiff} and \cite[Chapter 15]{villani1} for complementary material.

Let $\phi : \Pcal_2(\R^d) \rightarrow (-\infty,+\infty]$ be a lower semicontinuous and proper functional with effective domain $D(\phi) = \{ \mu \in \Pcal_2(\R^d) ~\text{s.t.}~ \phi(\mu) < +\infty \}$. We introduce in the following definition the concept of \textit{extended Fr\'echet subdifferential} in $(\Pcal_2(\R^d),W_2)$, following the terminology of \cite[Chapter 10]{AGS}.

\begin{Def}[Extended Wasserstein subdifferential]
\label{def:Subdifferentials}
Let $\mu \in D(\phi)$. We say that a transport plan $\Bgamma \in \Pcal_2(\R^{2d})$ belongs to the \textnormal{extended subdifferential} $\Bpartial \phi(\mu)$ of $\phi(\cdot)$ at $\mu$ provided that 
\begin{enumerate}
\item[(i)] $\pi^1_{\#} \Bgamma = \mu$.
\item[(ii)] For any $\nu \in \Pcal_2(\R^d)$, it holds that 
\begin{equation*} 
\phi(\nu) - \phi(\mu) ~ \geq ~ \inf\limits_{\Bmu \in \Gamma^{1,3}_o(\Bgamma,\nu)} \INTDom{\langle x_2 , x_3-x_1 \rangle}{\R^{3d}}{\Bmu(x_1,x_2,x_3)} + o(W_2(\mu,\nu)),
\end{equation*}
where $\Gamma^{1,3}_o(\Bgamma,\nu) = \{\Bmu \in \Gamma(\Bgamma,\nu) ~\text{s.t.}~ \pi^{1,3}_{\#} \Bmu \in \Gamma_o(\mu,\nu) \}$.
\end{enumerate}

We furthermore say that a transport plan $\Bgamma \in \Pcal_2(\R^{2d})$ belongs to the \textnormal{strong extended subdifferential} $\Bpartial_S \phi(\mu)$ of $\phi(\cdot)$ at $\mu$ if the stronger condition 
\begin{equation} 
\label{eq:Def_StrongSubdiff}
\phi(\nu) - \phi(\mu) ~ \geq ~ \INTDom{\langle x_2 , x_3-x_1 \rangle}{\R^{3d}}{\Bmu(x_1,x_2,x_3)} + o(W_{2,\Bmu}(\mu,\nu)),
\end{equation}
holds for any $\nu \in \Pcal_2(\R^d)$ and $\Bmu \in \Gamma(\Bgamma,\nu)$, where the quantity $W_{2,\Bmu}(\mu,\nu)$ is defined by
\begin{equation*}
W_{2,\Bmu}(\mu,\nu) = \left( \INTDom{|x_1-x_3|^2}{\R^{2d}}{\Bmu(x_1,x_2,x_3)} \right)^{1/2}.
\end{equation*}
\end{Def}

We proceed by recalling the notions of \textit{regularity} and \textit{metric slope} that are instrumental in deriving a sufficient condition for the extended subdifferential of a functional to be non-empty. This result is stated in Theorem \ref{thm:subdifferential_localslope} and its proof can be found in \cite[Theorem 10.3.10]{AGS}.

\begin{Def}[Regular functionals over $(\Pcal_2(\R^d),W_2)$ and metric slope]
\label{def:Regular}
A proper and lower semicontinuous functional $\phi(\cdot)$ is said to be \textnormal{regular} provided that whenever $(\mu_n) \subset \Pcal_2(\R^d)$ and $(\Bgamma_n) \subset \Pcal_2(\R^{2d})$ are taken such that
\begin{equation*}
\left\{
\begin{aligned}
& \mu_n \overset{W_2 \,}{\underset{n \rightarrow +\infty}{\longrightarrow}} \mu,  \hspace{0.7cm} \phi(\mu_n) \underset{n \rightarrow +\infty}{\longrightarrow} \tilde{\phi}, \\
& \Bgamma_n \in \Bpartial_S \phi(\mu_n), \qquad \Bgamma_n \overset{W_2 \,}{\underset{n \rightarrow +\infty}{\longrightarrow}} \Bgamma,
\end{aligned}
\right.
\end{equation*}
it implies that $\Bgamma \in \Bpartial \phi(\mu)$ and $\tilde{\phi} = \phi(\mu)$. Furthermore, we define the \textit{metric slope} $|\partial \phi|(\mu)$ of the functional $\phi(\cdot)$ at $\mu \in D(\phi)$ as
\begin{equation*}
| \partial \phi |(\mu) = \underset{\nu \rightarrow \mu}{\textnormal{limsup}} \left[ \frac{\left( \phi(\mu) - \phi(\nu) \right)^+}{W_2(\mu,\nu)}\right],
\end{equation*}
where $(\bullet)^+$ denotes the positive part.
\end{Def}

\begin{thm}[Link between extended subdifferentials and metric slopes]
\label{thm:subdifferential_localslope}
Let $\phi(\cdot)$ be a proper, bounded from below, lower semicontinuous and regular functional over $\Pcal_2(\R^d)$. Then, the extended subdifferential $\Bpartial \phi(\mu)$ of $\phi(\cdot)$ at some $\mu \in D(\phi)$ is non-empty if and only if $|\partial \phi|(\mu) < +\infty$.

In this case, there exists a unique \textnormal{minimal selection} in $\Bpartial \phi(\mu)$ -- denoted by $\Bpartial^{\circ} \phi(\mu)$ -- which satisfies
\begin{equation*}
\left( \INTDom{|r|^2}{\R^{2d}}{(\Bpartial^{\circ} \phi(\mu))(x,r)} \right)^{1/2} =  \min_{\Bgamma} \left\{ \left( \INTDom{|r|^2}{\R^{2d}}{\Bgamma(x,r)} \right)^{1/2} ~ \text{s.t.} ~ \Bgamma \in \Bpartial \phi(\mu) \right\} = |\partial \phi |(\mu).
\end{equation*}
This minimal selection can moreover be explicitly characterized as follows. Let $\mu_{\tau}$ be the minimizer of the \textnormal{Moreau-Yosida} functional 
\begin{equation}
\label{eq:MoreauYosida}
\Phi_{\textnormal{M}}(\mu,\tau;\cdot) : \nu \in \Pcal_2(\R^d) \mapsto \tfrac{1}{2 \tau} W_2^2(\mu,\nu) + \phi(\nu)
\end{equation}
at some $\tau \in (0,\tau_*)$ with $\tau_* > 0$ small. Then, there exists a family of strong subdifferentials $(\Bgamma_{\tau}) \subset (\Bpartial_S \phi(\mu_{\tau}))$ which converges towards $\Bpartial^{\circ} \phi(\mu)$ in the $W_2$-metric along a suitable vanishing sequence $\tau_n \downarrow 0$.
\end{thm}

The main reason for resorting to the abstract notion of measure subdifferentials in the context of our argument is that the approximation property of the minimal selection by a sequence of strong subdifferentials plays a key role in the proof of the general Wasserstein chainrule of Proposition \ref{prop:Chainrule} below. In the sequel however, we will mainly use the simpler notion of \textit{classical Wasserstein differentials} which we introduce in the following definition.

\begin{Def}[Classical Wasserstein subdifferentials and superdifferentials]
\label{def:ClassicalSubdifferentials}
Let $\mu \in D(\phi)$. We say that a map $\xi \in L^2(\R^d,\R^d;\mu)$ belongs to the \textnormal{classical subdifferential} $\partial^- \phi(\mu)$ of $\phi(\cdot)$ at $\mu$ provided that 
\begin{equation*} 
\phi(\nu) - \phi(\mu) \geq \sup\limits_{\gamma \in \Gamma_o(\mu,\nu)} \INTDom{\langle \xi(x) , y - x \rangle}{\R^{2d}}{\gamma(x,y)} + o(W_2(\mu,\nu))
\end{equation*}
for all $\nu \in \Pcal_2(\R^d)$. Similarly, we say that a map $\xi \in L^2(\R^d,\R^d;\mu)$ belongs to the \textnormal{classical superdifferential} $\partial^+ \phi(\mu)$ of $\phi(\cdot)$ at $\mu$ if $(-\xi) \in \partial^- (-\phi)(\mu)$. 
\end{Def}

It has been proven recently in \cite{WassDiff} that the definition of classical Wasserstein subdifferential involving a supremum taken over the set of optimal transport plans is equivalent to the usual one introduced in \cite{AGS} which involves an infimum. This allows for the elaboration of a convenient notion of differentiability in Wasserstein spaces as detailed below.

\begin{Def}[Differentiable functionals in $(\Pcal_2(\R^d),W_2)$]
\label{def:WassersteinDiff}
A functional $\phi : \Pcal_2(\R^d) \mapsto \R$ is said to be \textnormal{Wasserstein-differentiable} at some $\mu \in D(\phi)$ if $\partial^- \phi(\mu) \cap \partial^+ \phi(\mu) \neq \emptyset$. In this case, there exists a unique element $\nabla_{\mu} \phi(\mu) \in \partial^- \phi(\mu) \cap \partial^+ \phi(\mu)$ called the \textnormal{Wasserstein gradient} of $\phi(\cdot)$ at $\mu$, which satisfies
\begin{equation}
\label{eq:WassersteinDiff}
\phi(\nu) - \phi(\mu) = \INTDom{\langle \nabla_{\mu} \phi(\mu)(x) , y - x \rangle}{\R^{2d}}{\gamma(x,y)} + o(W_2(\mu,\nu)), 
\end{equation}
for any $\nu \in \Pcal_2(\R^d)$ and any $\gamma \in \Gamma_o(\mu,\nu)$. 
\end{Def}

\begin{rmk}[Consistency of Definition \ref{def:Subdifferentials} and Definition \ref{def:WassersteinDiff}]
\label{rmk:WassersteinDiff}
Let it be noted that the general statements of Definition \ref{def:Subdifferentials} and Theorem \ref{thm:subdifferential_localslope} are consistent with the simpler ones provided in Definition \ref{def:WassersteinDiff}. Indeed, if a functional $\phi : \Pcal_2(\R^d) \rightarrow \R$ is differentiable at some $\mu \in D(\phi)$, then the plan $(\Id \times \nabla_{\mu} \phi(\mu))_{\#} \mu$ is the minimal selection $\Bpartial^{\circ} \phi(\mu)$ in its extended Fr\'echet subdifferential at $\mu$. This follows from the fact that by definition $(\Id \times \nabla_{\mu}\phi(\mu))_{\#} \mu \in \Bpartial \phi(\mu)$, and from the estimate $\NormL{\nabla_{\mu}\phi(\mu)}{2}{\mu} \leq |\partial \phi|(\mu)$. The latter can be obtained by taking derivatives of $\phi(\cdot)$ along elements of the tangent space to $\Pcal_2(\R^d)$ (see \cite[Chapter 8,10]{AGS}) and then by using the fact that $\nabla_{\mu} \phi(\mu)(\cdot)$ is a tangent vector as shown in \cite[Theorem 3.10]{WassDiff}.
\end{rmk}

We conclude these recalls by stating in Proposition \ref{prop:Chainrule} below a new chainrule formula for Wasserstein-differentiable functionals along suitable multi-dimensional perturbations of a measure.

\begin{prop}[Chainrule along multidimensional perturbations by smooth vector fields]
\label{prop:Chainrule}
Let $K \subset \R^d$ be a compact set and $\mu \in \Pcal(K)$. Suppose that $\phi : \Pcal(K) \rightarrow  \R$ is Lipschitz in the $W_2$-metric, regular in the sense of Definition \ref{def:Regular} and Wasserstein-differentiable over $\Pcal(K)$. Given $N \geq 1$ and a small parameter $\epsilon > 0$, suppose that $\G \in C^0([-\epsilon,\epsilon]^N \times K , \R^d)$ is a map satisfying the following assumptions. 
\begin{enumerate}
\item[(i)] $\G(0,\cdot) = \Id$ and $e \mapsto \G(e,x)$ is Fr\'echet-differentiable at $e = 0$ uniformly with respect to $x \in K$.
\item[(ii)] $\supp(\G(e,\cdot)_{\#} \mu) \subset K$ for all $e \in [-\epsilon,\epsilon]^N$.
\item[(iii)] The directional derivative map
\begin{equation*}
\F_{\sigma}: x \in K \mapsto \D_e \G(0,x) \sigma = \sum_{k=1}^N \sigma_k \F_k(x)
\end{equation*} 
is continuous for all $\sigma \in [-\epsilon,\epsilon]^N$. 
\end{enumerate}
Then, the map $e \in [-\epsilon,\epsilon]^N \mapsto \phi(\G(e,\cdot)_{\#} \mu)$ is Fr\'echet-differentiable at $e = 0$ and 
\begin{equation*}
\nabla_e \big( \phi(\G(e,\cdot)_{\#} \mu) \big)_{\vert e=0}(\sigma) = \INTDom{\langle \nabla_{\mu} \phi(\mu)(x) , \D_e \G(0,x) \sigma \rangle}{\R^d}{\mu(x)} = \sum_{k=1}^N \sigma_k \INTDom{\langle \nabla_{\mu} \phi(\mu)(x) , \F_k(x) \rangle}{\R^d}{\mu(x)},
\end{equation*}
for any $\sigma \in [-\epsilon,\epsilon]^N$. 
\end{prop}

\begin{proof}
By assumption, $\phi(\cdot)$ is Lipschitz over $\Pcal(K)$ in the $W_2$-metric, regular in the sense of Definition \ref{def:Regular} and Wasserstein-differentiable in the sense of Definition \ref{def:WassersteinDiff} at $\mu$. Hence, $|\partial \phi|(\mu) < +\infty$ and $\Bpartial^{\circ} \phi(\mu) = (\Id \times \nabla_{\mu} \phi(\mu))_{\#} \mu$.

Let $(\tau_n) \subset (0,\tau_*)$ be a suitable vanishing sequence, $(\mu_{\tau_n}) \subset D(\phi)$ be a sequence of minimizers of the Moreau-Yosida functional \eqref{eq:MoreauYosida} and $(\Bgamma_{\tau_n}) \subset (\Bpartial_S \phi(\mu_{\tau_n}))$ be the corresponding sequence of strong subdifferentials converging in the $W_2$-metric towards $(\Id \times \nabla_{\mu} \phi(\mu))_{\#} \mu$. For any $\sigma \in [-\epsilon,\epsilon]^N$, define the family of 3-plans $\Bmu_{\sigma}^{\tau_n} = (\pi^1,\pi^2,\G(\sigma,\cdot) \circ \pi^1)_{\#} \Bgamma_{\tau_n}$. By \eqref{eq:Def_StrongSubdiff}, it holds that
\begin{equation*}
\frac{\phi(\G(\sigma,\cdot)_{\#} \mu_{\tau_n}) - \phi(\mu_{\tau_n})}{|\sigma|} \geq \INTDom{\Big\langle r , \frac{\G(\sigma,x)-x}{|\sigma|} \Big\rangle}{\R^{2d}}{\Bgamma_{\tau_n}(x,r)} + o(1), 
\end{equation*}
since $o(W_{2,\Bmu_\sigma^{\tau_n}}(\mu_{\tau_n},\G(\sigma,\cdot)_{\#}\mu_{\tau_n})) = o(\NormL{\G(\sigma,\cdot)-\Id}{2}{\mu_{\tau_n}}) = o(|\sigma|)$. For all $\sigma \in [-\epsilon,\epsilon]^N$, the maps $(x,r) \mapsto | \langle r , (\G(\sigma,x)-x)/|\sigma| \rangle |$ are continuous and uniformly integrable with respect to the family of measures $\{\Bgamma_{\tau_n} \}_{n=1}^{+\infty}$. By a classical convergence result for sequences of measures (see e.g. \cite[Lemma 5.1.7]{AGS}), we recover that
\begin{equation}
\label{eq:ChainruleEq1}
\begin{aligned}
\INTDom{\Big\langle r , \frac{\G(\sigma,x)-x}{|\sigma|} \Big\rangle}{\R^{2d}}{\Bgamma_{\tau_n}(x,r)} ~\underset{n \rightarrow +\infty}{\longrightarrow}~ & \INTDom{\Big\langle r , \frac{\G(\sigma,x)-x}{|\sigma|} \Big\rangle}{\R^{2d}}{(\Bpartial^{\circ} \phi(\mu))(x,r)} \\
= \hspace{0.45cm} & \INTDom{\Big\langle \nabla_{\mu} \phi(\mu)(x) , \frac{\G(\sigma,x)-x}{|\sigma|} \Big\rangle}{\R^d}{\mu(x)}.
\end{aligned}
\end{equation}
Furthermore, the continuity of $x \mapsto \G(\sigma,x)$ uniformly with respect to $\sigma \in [-\epsilon,\epsilon]^N$ together with the Lipschitzianity of $\phi(\cdot)$ with respect to the $W_2$-topology implies that
\begin{equation}
\label{eq:ChainruleEq2}
\frac{\phi(\G(\sigma,\cdot)_{\#} \mu_{\tau_n}) - \phi(\mu_{\tau_n})}{|\sigma|} ~\underset{n \rightarrow +\infty}{\longrightarrow}~ \frac{\phi(\G(\sigma,\cdot)_{\#} \mu_{\tau}) - \phi(\mu_{\tau})}{|\sigma|}.
\end{equation}

Merging together \eqref{eq:ChainruleEq1} and \eqref{eq:ChainruleEq2} and applying Vitali's Convergence Theorem to the $\mu$-uniformly integrable family of maps $\sigma \mapsto \langle \nabla_{\mu} \phi(\mu)(\cdot) , (\G(\sigma,\cdot)-\Id)/|\sigma| \rangle$ (see e.g. \cite[Exercise 1.18]{AmbrosioFuscoPallara}), we obtain that
\begin{equation*}
\frac{\phi(\G(\sigma,\cdot)_{\#} \mu) - \phi(\mu)}{|\sigma|} \geq \INTDom{\big\langle \nabla_{\mu} \phi(\mu)(x) , \D_e \G(0,x) \sigma \big\rangle}{\R^d}{\mu(x)} + o(1) = \sum_{k=1}^N \sigma_k \INTDom{\big\langle \nabla_{\mu} \phi(\mu)(x) , \F_k(x) \big\rangle}{\R^d}{\mu(x)} + o(1).
\end{equation*}
Using the fact that $(\Id \times (-\nabla_{\mu} \phi(\mu)))_{\#} \mu$ is the minimal selection in the extended Fr\'echet subdifferential of $(-\phi(\cdot))$ as a consequence of Remark \ref{rmk:WassersteinDiff}, we recover the converse inequality. 
\end{proof}

We list in Appendix \ref{appendix:Examples} a series of commonly encountered functionals which are both regular in the sense of Definition \ref{def:Regular} differentiable in the sense of Definition \ref{def:WassersteinDiff}, and we provide their Wasserstein gradients. This list was already presented along the same lines in our previous work \cite{PMPWass}, and we add it here for self-containedness.


\subsection{The continuity equation with non-local velocities in $\R^d$}

In this section, we introduce the continuity equation with non-local velocities in $(\Pcal_c(\R^d),W_1)$. This equation is commonly written as
\begin{equation} \label{eq:Nonlocal_continuity_equation}
\partial_t \mu(t) + \nabla \cdot \left( v[\mu(t)](t,\cdot) \mu(t) \right) = 0,
\end{equation}
where $t \mapsto \mu(t)$ is a narrowly continuous family of probability measures on $\R^d$ and $(t,x) \mapsto v[\mu](t,x)$ is a Borel family of vector fields satisfying the condition
\begin{equation}
\INTSeg{\INTDom{|v[\mu(t)](t,x)| \,}{\R^d}{\mu(t)(x)}}{t}{0}{T} < +\infty.
\end{equation}  
Equation \eqref{eq:Nonlocal_continuity_equation} has to be understood in duality with smooth and compactly supported functions, i.e.  
\begin{equation}
\label{eq:NonlocalPDE_distributions1}
\INTSeg{\INTDom{\Big( \partial_t \xi(t,x) + \left\langle \nabla_x \xi(t,x) , v[\mu(t)](t,x) \Big\rangle \right)}{\R^d}{\mu(t)(x)}}{t}{0}{T} = 0
\end{equation}
for all $\xi \in C^{\infty}_c([0,T] \times \R^d)$. This definition can be  alternatively written as
\begin{equation}
\label{eq:NonlocalPDE_distributions2}
\derv{}{t}{} \INTDom{\xi(x)}{\R^d}{\mu(t)(x)} = \INTDom{\langle \nabla \xi(x) , v[\mu(t)](t,x) \rangle}{\R^d}{\mu(t)(x)}
\end{equation}
for all $\xi \in C^{\infty}_c(\R^d)$ and $\Lcal^1$-almost every $t \in [0,T]$. 

We recall in Theorem \ref{thm:Nonlocal_PDE} the classical existence, uniqueness and representation formula for solutions of non-local PDEs. Although these results were first derived in \cite{AmbrosioGangbo}, we state here a version explored in \cite{Pedestrian,ControlKCS} which is better suited to our smoother control-theoretic framework. 

\begin{thm}[Existence, uniqueness and representation of solutions for \eqref{eq:Nonlocal_continuity_equation}] 
\label{thm:Nonlocal_PDE}
Consider a non-local velocity field $v[\cdot](\cdot,\cdot)$ defined as
\begin{equation}
v : \mu \in \Pcal_c(\R^d) \mapsto v[\mu](\cdot,\cdot) \in L^1(\R , C^1(\R^d,\R^d)), 
\end{equation}
and satisfying the following assumptions
\begin{framed}
\vspace{-0.2cm}
\begin{center}
\textnormal{\textbf{(H')}}
\end{center}
\begin{enumerate}
\item[(i)] There exist positive constants $L_1$ and $M$ such that
\begin{equation*}
| v[\mu](t,x) | \leq M (1+ |x|) ~~ \text{and} ~~ | v[\mu](t,x) - v[\mu](t,y) | \leq L_1 |x-y|
\end{equation*}
for every $\mu \in \Pcal_c(\R^d)$, $t \in \R$ and $(x,y) \in \R^{2d}$.
\item[(ii)] There exists a positive constant $L_2$ such that
\begin{equation*}
\NormC{v[\mu](t,\cdot) - v[\nu](t,\cdot)}{0}{\R^d} \leq L_2 W_1(\mu,\nu) 
\end{equation*}
for every $\mu,\nu \in \Pcal_c(\R^d)$ and $t \in \R$.
\end{enumerate}
\vspace{-0.2cm}
\end{framed}

Then, for every initial datum $\mu^0 \in \Pcal_c(\R^d)$, the Cauchy problem
\begin{equation} \label{eq:Nonlocal_Cauchy_problem}
\left\{
\begin{aligned}
& \partial_t \mu(t) + \nabla \cdot \left( v[\mu(t)](t,\cdot) \mu(t) \right) = 0 \\
& \mu(0) = \mu^0,
\end{aligned}
\right.
\end{equation}
admits a unique solution $\mu(\cdot) \in \Lip_{\loc}(\R_+,\Pcal_c(\R^d))$. If $\mu^0$ is absolutely continuous with respect to $\Lcal^d$, then $\mu(t)$ is absolutely continuous with respect to $\Lcal^d$ as well for all times $t \in \R_+$. Furthermore for every $T > 0$ and every $\mu^0,\nu^0 \in \Pcal_c(\R^d)$, there exists positive constants $R_T,L_T > 0$ such that
\begin{equation*}
\supp(\mu(t)) \subset \overline{B(0,R_T)} \qquad \text{and} \qquad W_1(\mu(t),\nu(t)) \leq L_T W_1(\mu^0,\nu^0)
\end{equation*}
for all times $t \in [0,T]$, where $\mu(\cdot),\nu(\cdot)$ are solutions of \eqref{eq:Nonlocal_Cauchy_problem} with initial conditions $\mu^0,\nu^0$ respectively.

Let $\mu(\cdot)$ be the unique solution of \eqref{eq:Nonlocal_Cauchy_problem} and $(\Phi_{(0,t)}^v[\mu^0](\cdot))_{t \geq 0}$ be the family of flows of diffeomorphisms generated by the non-autonomous velocity field $(t,x) \mapsto v[\mu(t)](t,x)$, i.e.
\begin{equation}
\label{eq:Flow_def}
\left\{
\begin{aligned}
\partial_t \Phi_{(0,t)}^v[\mu^0](x) & = v[\mu(t)] \left( t,\Phi_{(0,t)}^v[\mu^0](x) \right), \\
\Phi^v_{(0,0)}[\mu^0](x) & = x \hspace{1.25cm} \text{for all $x$ in $\R^d$}.
\end{aligned}
\right.
\end{equation}
Then, the curve $\mu(\cdot)$ is given explicitly by the pushforward formula
\begin{equation}
\label{eq:Representation_Formula}
\mu(t) = \Phi_{(0,t)}^v[\mu^0](\cdot)_{\#} \mu^0, 
\end{equation} 
for all times $t \in [0,T]$. 
\end{thm}

In the following proposition, we recall a standard result which links the differential in space of the flow of diffeomorphisms of an ODE to the solution of the corresponding linearized Cauchy problem (see e.g. \cite[Theorem 2.3.1]{BressanPiccoli}).

\begin{prop}[Classical differential of a non-local flow of diffeomorphisms]
\label{prop:Classical_Differential_flow}
Let $\mu^0 \in \Pcal_c(\R^d)$ and let $(\Phi^v_{(0,t)}[\mu^0](\cdot))_{t \in [0,T]}$ be the flows of diffeomorphisms generated by a non-local velocity field $v[\cdot](\cdot,\cdot)$ satisfying hypotheses \textnormal{\textbf{(H')}}. Then, the flow map $x \mapsto \Phi^v_{(s,t)}[\mu(s)](x)$ is Fr\'echet differentiable over $\R^d$ for any $s,t \in [0,T]$. Moreover, its differential $\D_x \Phi^v_{(s,t)}[\mu(s)](x) h$ in a direction $h \in \R^d$ is the unique solution of the linearized Cauchy problem
\begin{equation}
\label{eq:Differential_LocalFlow}
\left\{
\begin{aligned}
\partial_t w(t,x) & = \D_x v [\mu(t)] \left( t,\Phi^v_{(s,t)}[\mu(s)](x) \right) w(t,x), \\
w(s,x) & = h, 
\end{aligned}
\right.
\end{equation}
for any $s,t \in [0,T]$. 
\end{prop}

In our previous work \cite{PMPWass}, we extended the classical result of Proposition \ref{prop:Classical_Differential_flow} to the Wasserstein setting in order to compute derivatives of the flow maps $\Phi^v_{(0,t)}[\mu^0](\cdot)$ with respect to their initial measure $\mu^0$. In the following proposition, we state a further refinement of this result to the case in which the initial measure is perturbed by a multi-dimensional family of maps, in the spirit of the chainrule stated in Proposition \ref{prop:Chainrule}.

\begin{prop}[Wasserstein differential of a non-local flow of diffeomorphisms]
\label{prop:Wasserstein_Differential_flow}
Let $K \subset \R^d$ be a compact set, $\mu^0 \in \Pcal(K)$, $v[\cdot](\cdot,\cdot)$ be a non-local velocity field satisfying hypotheses \textnormal{\textbf{(H')}} of Theorem \ref{thm:Nonlocal_PDE} and $(\Phi^v_{(0,t)}[\mu^0](\cdot))_{t \in [0,T]}$ be its associated flow of diffeomorphisms. Suppose moreover that for any $i \in \{1,\dots,d\}$, the maps $\mu \mapsto v^i[\mu](t,x)$ are regular and differentiable over $\Pcal(K)$ uniformly with respect to $(t,x) \in [0,T] \times K$. Given $N \geq 1$ and a small parameter $\epsilon > 0$, let $\G \in C^0([-\epsilon,\epsilon]^N \times K , \R^d)$ be a function satisfying the hypotheses of Proposition \ref{prop:Chainrule}. 

Then, the map $e \in [-\epsilon,\epsilon]^N \mapsto \Phi^v_{(s,t)}[\G(e,\cdot)_{\#}\mu](x)$ is Fr\'echet-differentiable at $e = 0$ and its differential $w_{\sigma}(\cdot,x)$ in an arbitrary direction $\sigma \in [-\epsilon,\epsilon]^N$ can be expressed as
\begin{equation*}
w_{\sigma}(t,x) = \sum_{k=1}^N \sigma_k w_k(t,x),
\end{equation*}
where for any $k \in \{1,\dots,N\}$, the map $w_k(\cdot,x)$ is the unique solution of the non-local Cauchy problem
\begin{equation}
\label{eq:Differential_NonLocalFlow}
\left\{
\begin{aligned}
\partial_t w_k(t,x) & = \D_x v [\mu(t)] \left( t,\Phi^v_{(s,t)}[\mu](x) \right) w_k(t,x) \\
& + \INTDom{\BGamma^v_{\left( t , \Phi^v_{(s,t)}[\mu](x) \right)} \left( \Phi^v_{(s,t)}[\mu](y) \right) \left(\D_x \Phi^v_{(s,t)}[\mu](y) \F_k(y) + w_k(t,y) \right)}{\R^d}{\mu(y)} \\
w_k(s,x) & = 0.
\end{aligned}
\right.
\end{equation}
Here, $(t,x,y) \mapsto \BGamma^v_{(t,x)}(y) \in \R^{d \times d} $ is the matrix-valued map whose rows are the Wasserstein gradients of the components $v^i[\cdot](t,x)$ of the non-local velocity field at $\mu(t)$, i.e.
\begin{equation}
\label{eq:BGamma_Def}
\left( \BGamma^{v}_{(t,x)}(y) \right)_{i,j} = \Big( \nabla_{\mu} \left( v^i[\cdot](t,x) \right)(\mu(t))(y) \Big)_j,
\end{equation}
for any $i,j \in \{1,\dots,d\}$. 
\end{prop}

\begin{proof}
By Proposition \ref{prop:Chainrule} and as a consequence of our hypotheses on $v[\cdot](\cdot,\cdot)$, we know that the map $e \in [-\epsilon,\epsilon]^N \mapsto \Phi^{v}_{(s,t)}[\G(e,\cdot)_{\#} \mu](x)$ is Fr\'echet-differentiable at $e=0$. Therefore, the action of its differential on a given direction $\sigma \in [-\epsilon,\epsilon]^N$ can be expressed in coordinates using partial derivatives, i.e.
\begin{equation*}
w_{\sigma}(t,x) = \sum_{k=1}^N \sigma_k \partial_{e_k} \left( \Phi^v_{(s,t)}[\G(e,\cdot)_{\#}\mu](x) \right)(0). 
\end{equation*}
Moreover, it has been proven in \cite[Proposition 5]{PMPWass} that such one-dimensional variations could be characterized as the unique solution of the linearized Cauchy problems \eqref{eq:Differential_NonLocalFlow}. 
\end{proof}


\subsection{Non-smooth multiplier rule and differentiable extension of functions}
\label{subsection:MP_Multiplier}

In this section, we recall some facts of non-smooth analysis as well as a non-smooth Lagrange multiplier rule which is instrumental in the proof of our main result. This multiplier rule is expressed in terms of the so-called \textit{Michel-Penot subdifferential}, see e.g. \cite{MichelPenot,Ioffe}. In the sequel, we denote by $(X, \Norm{\cdot}_X)$ a separable Banach space and by $X^*$ its topological dual associated with the duality bracket $\langle \cdot,\cdot \rangle_X$. Given a map $f : X \rightarrow \R$, we denote by $D(f) = \{ x \in X ~\text{s.t.}~ f(x) < +\infty \}$ its effective domain. 

\begin{Def}[Michel-Penot subdifferential]
Given a map $f : X \rightarrow \R$, the \textnormal{Michel-Penot subdifferential} (MP-subdifferential in the sequel) of $f(\cdot)$ at some $x \in D(f)$ is defined by 
\begin{equation*}
\partial_{\MP} f (x) = \Big\{ \xi \in X^* ~\text{s.t.}~ \langle \xi , h \rangle_X \leq d_{\MP}f(x;h) ~ \text{for all $h \in X$} \Big\},
\end{equation*}
where 
\begin{equation*}
d_{\MP}f(x \, ;h) = \sup\limits_{e \in X} \limsup\limits_{t \downarrow 0} \left[ \frac{f(x+t(e+h)) - f(x+te)}{t} \right]
\end{equation*}
denotes the so-called \textnormal{Michel-Penot derivative} of $f(\cdot)$ at $x$ in the direction $h$. Moreover, if $f : X \rightarrow \R$ is locally convex around $x \in X$, then its Michel-Penot and convex subdifferentials coincide, i.e. $\partial_{\MP} f(x) = \partial f(x)$.
\end{Def}

The MP-subdifferential -- smaller than the Clarke subdifferential -- bears the nice property of shrinking to a singleton whenever the functional $f(\cdot)$ is merely Fr\'echet-differentiable. It also enjoys a summation rule and a chained-derivative formula for compositions of locally Lipschitz  and Fr\'echet-differentiable maps. We list these properties in the following proposition.

\begin{prop}[Properties of the Michel-Penot subdifferentials]
\label{prop:Prop_MPsubdiff}
Let $x \in X$, $f,g : X \rightarrow \R$ and $\G : \R^N \rightarrow X$. 
\begin{enumerate}
\item[(a)] If $f(\cdot)$ is Fr\'echet-differentiable at $x$, then $\partial_{\MP}f(x) = \{ \nabla f (x) \}$.
\item[(b)] If $d_{\MP} f(x \, ;h) < +\infty$ and $d_{\MP} g(x \, ;h) < +\infty$ for any $h \in X$, it holds that
\begin{equation*}
\partial_{\MP} (f+g)(x) \subseteq \Big( \partial_{\MP}f(x) + \partial_{\MP}g(x) \Big).
\end{equation*}
\item[(c)] If $\G(\cdot)$ is Fr\'echet-differentiable at $0 \in \R^N$ and $f(\cdot)$ is Lipschitz in a neighbourhood of $\G(0)$, one has that
\begin{equation}
\label{eq:MP_Inclusion}
d_{\MP} (f \circ \G) (0 \, ;\sigma) = \langle \xi , \D \G(0) \sigma \rangle_X, 
\end{equation}
for some $\xi \in \partial_{\MP} f(\G(0))$. In other words, $\partial_{\MP}(f \circ \G)(0) \subseteq \D \G(0)^* \circ \partial_{\MP} f(\G(0))$.
\end{enumerate}
\end{prop} 

These properties can be proven easily by computing explicitly the Michel-Penot derivatives of the corresponding maps and using the definition of the set $\partial_{\MP}(\bullet)$, see e.g. \cite{NonsmoothPMPNeedle2}. Another useful feature of this notion of subdifferential is that it allows to write Lagrange multiplier rules for locally Lipschitz functions. This family of optimality conditions was initially derived in \cite{Ioffe} and refined in \cite{NonsmoothPMPNeedle2} where the author extended the result to the class of so-called \textit{calm} functions. 

\begin{Def}[Calm functions]
A map $f : X \rightarrow \R$ is \textnormal{calm} at $x \in X$ provided that the following holds. 
\begin{enumerate}
\item[(i)] There exists a constant $L > 0$ such that for any $\delta \in X$ with $\Norm{\delta}_X$ sufficiently small, it holds that 
\begin{equation*}
\Norm{f(x+\delta) - f(x)}_X \, \leq \, L \Norm{\delta}_X.
\end{equation*}
\item[(ii)] $d_{\MP} f(x \,; h) < +\infty$ for any $h \in X$.
\end{enumerate}
\end{Def}

\begin{thm}[Multiplier rule for Michel-Penot subdifferentials]
\label{thm:Lagrange_MP}
Let $f_0,\dots,f_n,g_1,\dots,g_m : X \rightarrow \R$ and $\Omega \subset X$ be a closed and convex set. Suppose that $x_*$ is a local solution of the non-linear optimization problem
\begin{equation*}
\left\{
\begin{aligned}
\min\limits_{x \in \Omega} & \, \big[ f_0(x) \big] \\
\text{s.t.} ~ & \left\{ 
\begin{aligned}
& f_i(x) \leq 0 ~~ & \text{for all $i \in \{1,\dots,n\}$}, \\
& g_j(x) = 0 ~~ & \text{for all $j \in \{1,\dots,m\}$}, \\
\end{aligned}
\right.
\end{aligned}
\right.
\end{equation*}
and that the maps $f_0(\cdot),\dots,f_n(\cdot),g_1(\cdot),\dots,g_m(\cdot)$ are calm at $x_*$. Then, there exist Lagrange multipliers $(\lambda_0,\dots,\lambda_n,\eta_1,\dots,\eta_m) \in \{0,1\} \times \R_+^n \times \R^m$ such that the following \textit{stationarity} \textnormal{(S)}, \textit{non-triviality} \textnormal{(NT)} and \textit{complementary-slackness} \textnormal{(CS)} conditions hold
\begin{equation*}
\left\{
\begin{aligned}
& 0 \in \partial_{\MP} \Big( \lambda_0 f_0(\cdot) + \sum\limits_{i = 1}^n \lambda_i f_i(\cdot) + \sum\limits_{j = 1}^m \eta_j g_j(\cdot) \Big)(x_*) + \Ncal(\Omega,x_*) , ~~ & \textnormal{(S)} \\
& \lambda_0 + \sum\limits_{i = 1}^n \lambda_i + \sum\limits_{j = 1}^m |\eta_j| = 1, ~~ & \textnormal{(NT)} \\
& \lambda_i f_i(x_*) = 0 ~~ \text{for all $i \in \{ 1,\dots,n\}$}, ~~ & \textnormal{(CS)} \\
\end{aligned}
\right.
\end{equation*}
where $\Ncal(\Omega,x_*)$ denotes the normal cone of convex analysis to $\Omega$ at $x_*$.
\end{thm}

We end this introductory section by stating a Lusin-type lemma for vector-valued functions and a derivative-preserving continuous extension result that will both prove to be useful in the sequel. We refer the reader e.g. to \cite{DiestelUhl} for notions on Bochner integrals and abstract integration in Banach spaces. 

\begin{lem}[Pointwise convergence and restriction]
\label{lem:Pointwise_Convergence}
Let $f : [0,T] \rightarrow X$ be an $L^1$-function in the sense of Bochner and $\T$ be any subset of $[0,T]$ with full Lebesgue measure. Then, there exist $\Acal,\Mcal \subset \T$ respectively with null and full Lebesgue measure  satisfying the property that for any $\tau \in \Mcal$, there exists $(\tau_k) \subset \Acal$ such that 
\begin{equation*}
\tau_k \underset{k \rightarrow + \infty}{\longrightarrow} \tau \qquad \text{and} \qquad \|f(\tau) - f(\tau_k) \|_X \underset{k \rightarrow +\infty}{\longrightarrow} 0.
\end{equation*}
\end{lem}

\begin{proof}
This result is a consequence of \cite[Lemma 4.1]{NonsmoothPMPNeedle2} along with Lusin's Theorem for vector valued maps.
\end{proof}

\begin{lem}[A continuous extension preserving the derivative]
\label{lem:Continuous_Extension}
Let $\epsilon > 0$ and $f : [0,\epsilon]^N \rightarrow X$ be a continuous map  differentiable at $e = 0$ relatively to $\R^N_+$. Then, there exists a continuous extension $\tilde{f} : [-\bar{\epsilon}_N,\bar{\epsilon}_N]^N \rightarrow X$ of $f(\cdot)$ which is Fr\'echet-differentiable at $e = 0$ and such that $\D_e \tilde{f}(0) = \D_e f(0)$. 
\end{lem}

\begin{proof}
We adapt here a simple proof that can be found e.g. in \cite[Lemma 2.11]{NonsmoothPMPNeedle2}. Define the map
\begin{equation*}
g : e \in \R_+^N \mapsto \frac{1}{|e|} \Big(f(e) - f(0) - \D_e f(0) e \Big) \in X.
\end{equation*}
By definition, $g(\cdot)$ is continuous over $\R_+^N \backslash \{0\}$ and can be extended to $\R_+^N$ by imposing that $g(0)=0$ since $f(\cdot)$ is differentiable at $e = 0$ relatively to $\R^N_+$. Invoking Dugundji's extension theorem (see \cite{Dugundji}), we can define a continuous extension $\tilde{g}(\cdot)$ of $g(\cdot)$ on the whole of $\R^N$. 

We now define the auxiliary map $\tilde{f} : e \in \R^N \mapsto f(0) + \D_e f(0) e + |e| \tilde{g}(e)$. By construction, $\tilde{f}(\cdot)$ is continuous and coincides with  $f(\cdot)$ over $\R_+^N$. Moreover, one has for any $e \in \R^N$ that 
\begin{equation*}
\tilde{f}(e) - \tilde{f}(0) = \D_e f(0) e + |e| \tilde{g}(e) = \D_e f(0) e + o(|e|)
\end{equation*}
by continuity of $\tilde{g}(\cdot)$ at $0$. Therefore, the map $\tilde{f}(\cdot)$ is Fr\'echet-differentiable at $e = 0$ with $\D_e \tilde{f}(0) = \D_e f(0)$.  
\end{proof}


\section{Proof of the main result}
\label{section:PMP_General}
 
In this section we prove the main result of this article, that is Theorem \ref{thm:PMP_General} below. We recall that this result is a first-order necessary optimality condition of Pontryagin-type for the general Wasserstein optimal control problem $(\Ppazo)$ defined in the Introduction.

\begin{thm}[Pontryagin Maximum Principle for $(\Ppazo)$]
\label{thm:PMP_General}
Let $(u^*(\cdot),\mu^*(\cdot)) \in \U \times \Lip([0,T],\Pcal_c(\R^d))$ be an optimal pair control-trajectory for $(\Ppazo)$ and assume that the set of hypotheses \textnormal{\textbf{(H)}} below holds. 
\begin{framed}
\begin{center}
\textnormal{\textbf{(H)}}
\end{center}
\vspace{0.2cm}

\begin{enumerate}
\item[\textnormal{\textbf{(H1)}}] The set of admissible controls is defined as $\U = L^{\infty}([0,T],U)$ where $U$ is any $C^0$-closed subset of $\big\{ \omega \in C^1(\R^d,\R^d) ~\text{s.t.}~ \NormC{\omega(\cdot)}{0}{\R^d,\R^d} + \Lip(\omega(\cdot);\R^d) \leq L_U \big\}$ for a given constant $L_U > 0$. \\
\item[\textnormal{\textbf{(H2)}}] The non-local velocity field $\mu \in \Pcal_c(\R^d) \mapsto v[\mu] \in L^1([0,T],C^1(\R^d,\R^d))$ satisfies the classical Cauchy-Lipschitz assumptions in Wasserstein spaces, i.e. there exist positive constants $M,L_1$ and $L_2$ such that
\begin{equation*}
\begin{aligned}
|v[\mu](t,x)| \leq & ~ M(1+|x|) ~,~ |v[\mu](t,x) - v[\mu](t,y)| \leq L_1 |x-y|, \\
& \NormC{v[\mu](t,\cdot) - v[\nu](t,\cdot)}{0}{\R^d} \leq L_2 W_1(\mu,\nu)
\end{aligned}
\end{equation*}
for all $(x,y) \in \R^d$ and $\Lcal^1$-almost every $t \in [0,T]$. \\
\item[\textnormal{\textbf{(H3)}}] The maps $\mu \mapsto v^i[\mu](t,x) \in \R$ are regular in the sense of Definition \ref{def:Regular} over $\Pcal(K)$ for any compact set $K \subset \R^d$ and Wasserstein-differentiable at $\mu^*(t)$, uniformly with respect to $(t,x) \in [0,T] \times \R^d$. The maps $\mu \mapsto \D_x v[\mu](t,x) \in \R^d$ and $(y,z) \mapsto \nabla_{\mu} \left( v^i[\cdot](t,y) \right)(\mu)(z) \in \R^d$ are continuous for all $i \in \{1,\dots,d\}$, uniformly with respect to $(t,x) \in [0,T] \times \R^d$.
\\
\item[\textnormal{\textbf{(H4)}}] The final cost $\mu \mapsto \varphi(\mu) \in \R$ and the boundary constraints maps $\mu \mapsto (\Psi^I_i(\mu),\Psi^E_j(\mu))_{i,j} \in \R^{n + m}$ are bounded and Lipschitz continuous in the $W_2$-metric over $\Pcal(K)$ for any compact set $K \subset \R^d$. They are furthermore regular in the sense of Definition \ref{def:Regular} and Wasserstein-differentiable at $\mu^*(T)$. The Wasserstein gradients $\nabla_{\mu} \varphi(\mu^*(T))(\cdot)$, $\nabla_{\mu}\Psi^I_i(\mu^*(T))(\cdot)$ and $\nabla_{\mu} \Psi^E_j(\mu^*(T))(\cdot)$ are continuous for any $(i,j) \in \{1,\dots,n \} \times \{1,\dots,m\}$. \\
\item[\textnormal{\textbf{(H5)}}] The running cost $(t,\mu,\omega) \mapsto L(t,\mu,\omega) \in \R$ is $\Lcal^1$-measurable with respect to $t \in [0,T]$, bounded, and Lipschitz continuous with respect to the product $W_2 \times C^0$-metric defined over $\Pcal(K) \times U$ for any compact set $K \subset \R^d$. It is furthermore regular in the sense of Definition \ref{def:Regular} and Wasserstein-differentiable at $\mu^*(t)$ for any $(t,\omega) \in [0,T] \times  U$, and its Wasserstein gradient $\nabla_{\mu}L(t,\mu^*(t),\omega)(\cdot)$ is continuous. \\
\item[\textnormal{\textbf{(H6)}}] The state constraints maps $(t,\mu) \mapsto \Lambda_l(t,\mu) \in \R$ are bounded and Lipschitz-continuous over $[0,T] \times \Pcal(K)$ and regular in the sense of Definition \ref{def:Regular}, for any compact set $K \subset \R^d$ and any $l \in \{1,\dots,r\}$. Moreover, the maps $(t,\mu) \mapsto \partial_t \Lambda_l(t,\mu)$ and $(t,\mu) \mapsto \nabla_{\mu} \Lambda_l(t,\mu)(\cdot)$ are well-defined and continuously differentiable at $(t,\mu^*(t))$ with
\begin{equation*}
\begin{aligned}
\nabla_{\mu} \partial_t \Lambda_l(t,\mu^*(t))(\cdot) = & \partial_t \nabla_{\mu} \Lambda_l(t,\mu^*(t))(\cdot) \in C^0(K,\R^d) ~,~ \D_x \nabla_{\mu} \Lambda_l(t,\mu^*(t))(\cdot) \in C^0(K,\R^d), \\
& \nabla_{\mu} \left[ \nabla_{\mu} \Lambda_l(t,\mu^*(t))(\cdot) \right](\cdot) \in C^0(K^2,\R^d)
\end{aligned}
\end{equation*}
\end{enumerate}
\vspace{-0.3cm}
\end{framed}
Then there exists a constant $R'_T > 0$, Lagrange multipliers $(\lambda_0,\dots,\lambda_n,\eta_1,\dots,\eta_m,\varpi_1,\dots,\varpi_r) \in \{0,1\} \times \R_+^n \times \R^m \times \M_+([0,T])^r$ and a curve $\nu^*(\cdot) \in \Lip([0,T],\Pcal(\overline{B_{2d}(0,R'_T)}))$ such that the following holds. \\
\begin{enumerate}
\item[(i)] The map $t \mapsto \nu^*(t)$ is a solution of the \textnormal{forward-backward Hamiltonian system} of continuity equations
\begin{equation}
\label{eq:PMP_HamiltonianFlow_General}
\left\{
\begin{aligned}
& \partial_t \nu^*(t) + \nabla \cdot \big( \J_{2d} \nabla_{\nu} \Hcal_{\lambda_0}(t,\nu^*(t),\zeta^*(t),u^*(t))(\cdot,\cdot)\nu^*(t) \big) = 0 ~~ & \text{in $[0,T] \times \R^{2d}$}, \\
& \pi^1_{\#} \nu^*(t) = \mu^*(t) ~~ & \text{for all $t \in [0,T]$}, \\
& \pi^2_{\#} \nu^*(T) = (-\nabla_{\mu} \Scal(\mu^*(T)))_{\#} \mu^*(T),
\end{aligned}
\right.
\end{equation}
where $\J_{2d}$ is the symplectic matrix of $\R^{2d}$. The \textnormal{augmented infinite-dimensional Hamiltonian} $\Hcal_{\lambda_0}(\cdot,\cdot,\cdot,\cdot)$ of the system is defined by 
\begin{equation}
\label{eq:Hamiltonian_General}
\Hcal_{\lambda_0}(t,\nu,\zeta,\omega) = \INTDom{\langle r , v[\pi^1_{\#} \nu](t,x) + \omega(x) \rangle}{\R^{2d}}{\nu(x,r)} - \lambda_0 L(t,\pi^1_{\#} \nu,\omega) - \Ccal(t,\pi^1_{\#} \nu,\zeta,\omega),
\end{equation}
for any $(t,\nu,\zeta,\omega) \in [0,T] \times \Pcal(\overline{B_{2d}(0,R'_T)}) \times \R^r \times U$. The \textnormal{penalized state constraints} and \textnormal{final gradient} maps are given respectively by
\begin{equation}
\label{eq:Penalized_Constraints}
\Ccal(t,\mu,\zeta,\omega) = \sum_{l=1}^r \zeta_l \left(  \partial_t \Lambda_l(t,\mu) + \INTDom{\left\langle  \nabla_{\mu} \Lambda_l(t,\mu)(x) , v[\mu](t,x) + \omega(x) \right\rangle}{\R^d}{\mu(x)} \right),
\end{equation}
and
\begin{equation}
\label{eq:FinalGradient_General}
\nabla_{\mu} \Scal(\mu) = \lambda_0 \nabla_{\mu} \varphi(\mu) + \sum_{i=1}^n \lambda_i \nabla_{\mu} \Psi_i^I(\mu) + \sum_{j=1}^m \eta_j \nabla_{\mu} \Psi_j^E(\mu).
\end{equation}
For any $l \in \{1,\dots,r\}$, the map $t \in [0,T] \mapsto \zeta^*_l(t) \in \R_+$ denotes the \textnormal{cumulated state constraints multiplier} associated with $\varpi_l$, defined by
\begin{equation*}
\zeta^*_l(t) = \mathds{1}_{[0,T)}(t) \INTSeg{}{\varpi_l(s)}{t}{T},
\end{equation*}

\item[(ii)] The Lagrange multipliers satisfy the \textnormal{non-degeneracy} condition
\begin{equation}
\label{eq:NonDegeneracy_General}
(\lambda_0,\dots\lambda_n,\eta_1,\dots,\eta_m,\varpi_1,\dots,\varpi_r) \neq 0.
\end{equation}
as well as the \textnormal{complementary slackness} conditions
\begin{equation*}
\label{eq:Slackness_General}
\left\{
\begin{aligned}
& \lambda_i \Psi^I_i(\mu^*(T)) = 0 ~~ & \text{for all $i \in \{1,\dots,n\}$}, \\
& \supp(\varpi_l) = \Big\{ t \in [0,T] ~\text{s.t.}~ \Lambda_l(t,\mu(t)) = 0 \Big\} ~~ & \text{for all $l \in \{1,\dots,r\}$}.
\end{aligned}
\right.
\end{equation*}
\item[(iii)] The \textnormal{Pontryagin maximization condition}
\begin{equation}
\label{eq:Maximization_General}
\Hcal_{\lambda_0}(t,\nu^*(t),\zeta^*(t),u^*(t)) = \max_{\omega \in U} \big[ \Hcal_{\lambda_0}(t,\nu^*(t),\zeta^*(t),\omega) \big]
\end{equation}
holds for $\Lcal^1$-almost every $t \in [0,T]$. 
\end{enumerate}
\end{thm}

\begin{rmk}[The Gamkrelidze Maximum Principle]
The so-called Gamkrelidze formulation of the PMP corresponds to the case in which one includes the derivative of the state constraints inside the Hamiltonian function. The consequence of this choice is that the costate variables are absolutely continuous in time instead of being merely BV as it is the case in the more classical formulation of the constrained PMP (see e.g. \cite[Chapter 9]{Vinter}). 

As already mentioned in the Introduction, this fact is crucial for our purpose since absolutely continuous curves in Wasserstein spaces are exactly those curves which solve a continuity equation. Hence, one cannot derive a Hamiltonian system such as \eqref{eq:PMP_HamiltonianFlow_General} by sticking to the classical formulation of the PMP.
\end{rmk}

\begin{rmk}[On the regularity hypothesis \textnormal{\textbf{(H1)}}]
One of the distinctive features of continuity equations in Wasserstein spaces, compared to other families of PDEs is that they require Cauchy-Lipschitz assumptions on the driving velocity fields in order to be classically well-posed for arbitrary initial data. Even though the existence theory has gone far beyond this basic setting, notably through the DiPerna-Lions-Ambrosio theory (see \cite{DiPernaLions,AmbrosioPDE} or \cite[Section 8.2]{AGS}), such extensions come at the price of losing the strict micro-macro correspondence of the solutions embodied by the underlying flow-structure. Therefore, from a mean-field control-theoretic viewpoint, it seemed more meaningful for the author to work in a setting where classical well-posedness holds for the optimal trajectory.

Furthermore, the proof of Theorem \ref{thm:PMP_General} relies heavily on the geometric flow of diffeomorphism structure of the underlying characteristic system of ODEs, both forward and backward in time. For this reason, the Lipschitz-regularity assumption \textnormal{\textbf{(H1)}} is instrumental in our argumentation. 

Let it be remarked however that there exists common examples of Wasserstein optimal control problems for which the optimal control is $C^1$-smooth in space. Such a situation is given e.g. by controlled vector fields of the form $u(t,x) = \sum_{k=1}^m u_k(t) X_k(\cdot)$ where $X_1,\dots,X_m \in C^1(\R^d,\R^d)$ and $u_1,\dots,u_m \in L^{\infty}([0,T],\R)$, or by non-linear controlled vector field $(t,x,\mu,u) \in [0,T] \times \R^d \times \Pcal_c(\R^d) \times \R^m \mapsto v[\mu](t,x,u) \in \R^d$. 
\end{rmk}

We divide the proof of Theorem \ref{thm:PMP_General} into four steps. In Step 1, we introduce the concept of \textit{packages of needle-like variations} of an optimal control and compute the corresponding perturbations induced on the optimal trajectory. In Step 2, we apply the non-smooth Lagrange multiplier rule of Theorem \ref{thm:Lagrange_MP} to the sequence of finite-dimensional optimization problem formulated on the length of the needle variations, to obtain a family of finite-dimensional optimality conditions at time $T$. We introduce in Step 3 a suitable notion of costate, allowing to propagate this family of optimality condition backward in time, yielding the PMP with a relaxed maximization condition restricted to a countable subset of needle parameters. The full result is then recovered in Step 4 through a limiting procedure combined with several approximation arguments.


\bigskip

\begin{flushleft}
\underline{\textbf{Step 1 : Packages of needle-like variations :}}
\end{flushleft}

\bigskip

We start by considering an optimal pair control-trajectory $(u^*(\cdot),\mu^*(\cdot)) \in \U \times \Lip([0,T],\Pcal(\overline{B(0,R_T)}))$ where $R_T > 0$ is given by Theorem \ref{thm:Nonlocal_PDE}. Let $\T \subset [0,T]$ be the set of Lebesgue points of $t \in [0,T] \mapsto (v[\mu^*(t)](t,\cdot),u^*(t,\cdot)$ $,L(t,\mu^*(t),\cdot) \in C^0(\overline{B(0,R_T)},\R^d) \times U \times C^0(U,\R)$ in the sense of Bochner's integral (see e.g. \cite[Theorem 9]{DiestelUhl}). This set has full Lebesgue measure in $[0,T]$, and Lemma \ref{lem:Pointwise_Convergence} yields the existence of two subsets $\Acal,\Mcal \subset \T$, having respectively null and full Lebesgue measure, such that for any $\tau \in \Mcal$, there exists $(\tau_k) \subset \Acal$ converging towards $\tau$ and such that 
\begin{equation*}
\NormC{u^*(\tau,\cdot)-u^*(\tau_k,\cdot)}{0}{\R^d,\R^d} ~\underset{k \rightarrow +\infty}{\longrightarrow}~ 0 \qquad \text{and} \qquad \NormC{L(\tau,\mu^*(\tau),\cdot) - L(\tau_k,\mu^*(\tau_k),\cdot)}{0}{U,\R} ~\underset{k \rightarrow +\infty}{\longrightarrow}~ 0.
\end{equation*}
We further denote by $U_D$ a countable and dense subset of the set of admissible control values $U$ which is compact and separable in the $C^0$-topology as a consequence of  \textbf{(H1)}.

\begin{Def}[Package of needle-like variations]
\label{def:Needle_Variations}
Let $(u^*(\cdot),\mu^*(\cdot)) \in \U \times \Lip([0,T],\Pcal_c(\R^d))$ be an optimal pair control-trajectory. Given $N \geq 1$, a family of elements $\{(\omega_k , \tau_k)\}_{k = 1}^N \subset U_D \times \Acal$ and $e = (e_1,\dots,e_N) \in [0,\overline{\epsilon}_N]^N$ such that $[\tau_i-e_i,\tau_i] \cap [\tau_j-e_j,\tau_j] = \emptyset$ for all distinct pairs $i,j \in \{1,\dots,N\}$, we define the \textnormal{$N$-package of needle-like variations} of $u^*(\cdot)$ by
\begin{equation}
\label{eq:Needle_Variation}
\tilde{u}^N_e \equiv \tilde{u}_e : t \mapsto \left\{ 
\begin{aligned}
& \omega_k ~\hspace{0.3cm} \text{if $t \in [\tau_k-e_k,\tau_k]$,} \\
& u^*(t) \hspace{1.3cm} \text{otherwise}.
\end{aligned}
\right.
\end{equation}
We also denote by $t \mapsto \tilde{\mu}_e(\cdot)$ the corresponding perturbed trajectory, i.e. the solution of \eqref{eq:Nonlocal_Cauchy_problem} associated with the controlled non-local velocity field $v[\cdot](\cdot,\cdot) + \tilde{u}_e(\cdot,\cdot)$. 
\end{Def}

This class of variations is known in the literature of control theory to generate admissible perturbations of the optimal control without any assumption on the structure of the control set $\U$, while allowing for an explicit and tractable computation of the relationship between the perturbed and optimal states (see e.g. \cite{BressanPiccoli}).  

In the following lemma, we make use of the geometric structure of solutions to non-local transport equations presented in Theorem \ref{thm:Nonlocal_PDE} together with some notations borrowed from Proposition \ref{prop:Wasserstein_Differential_flow} to express $\tilde{\mu}_e(t)$ as a function of $\mu^*(t)$ for all times $t \in [0,T]$. In the sequel, we denote by $\Phi^{v,u}_{(s,t)}[\mu(s)](\cdot)$ the flow map generated by the non-local velocity field $v[\cdot](\cdot,\cdot) + u(\cdot,\cdot)$ between times $s$ and $t$, defined as in \eqref{eq:Flow_def}.

\begin{lem}[First-order perturbation induced by a package of needle-like variations in the non-local case]
\label{lem:First_order_needle_General}
There exists a family of maps $(\G^N_t(\cdot,\cdot))_{t \in [0,T]} \subset C^0([-\bar{\epsilon}_N,\bar{\epsilon}_N]^N \times \R^d,\R^d)$ such that 
\begin{enumerate}
\item[(i)] For all $(t,e) \in [0,T] \times [0,\bar{\epsilon}_N]^N$, the perturbed measures $\tilde{\mu}_e(t)$ satisfy
\begin{equation}
\label{eq:Perturbed_FinalState_General}
\tilde{\mu}_e(t) = \G^N_t(e,\cdot)_{\#}\mu^*(t).
\end{equation}
\item[(ii)] For all $(t,e) \in [0,T] \times [0,\bar{\epsilon}_N]^N$, the maps $\G_t^N(e,\cdot)$ are $C^1$-diffeomorphisms over $\overline{B(0,R_T)}$.  
\item[(iii)] There exists a constant $R_T^{\Phi} > 0$ depending on $R_T,L_U$ such that for all $(t,e) \in [0,T] \times [-\bar{\epsilon}_N,\bar{\epsilon}_N]^N$ one has $\supp(\G^N_t(e,\cdot)_{\#}\mu^*(T)) \subset \overline{B(0,R_T^{\Phi})}$
\item[(iv)] The map $e \in [-\bar{\epsilon}_N,\bar{\epsilon}_N]^N \mapsto \G^N_t(e,\cdot)$ is Fr\'echet-differentiable at $e = 0$ with respect to the $C^0(\overline{B(0,R_T)},\R^d)$ -norm, uniformly with respect to $t \in [0,T]$. The corresponding Taylor expansion can be written explicitly as
\begin{equation}
\label{eq:TaylorExpansion_Needle_General}
\G^N_t(e,\cdot) = \Id + \sum\limits_{k=1}^{\iota(t)} e_l \F^{\omega_k,\tau_k}_t \circ \Phi^{v,u^*}_{(t,\tau_k)}[\mu^*(t)](\cdot) + o(|e|),
\end{equation}
where $\iota(t) \in \{1,\dots,N\}$ is the biggest index such that $\tau_{\iota(t)} \leq t \leq \tau_{\iota(t)+1}-e_{\iota(t)+1}$. For all $x \in \overline{B(0,R_T)}$ and $k \in \{1,\dots,N\}$, the map $t \in [\tau_k,T] \mapsto \F^{\omega_k,\tau_k}_t(x)$ is the unique solution of the non-local Cauchy problem
\begin{equation}
\label{eq:FirstOrder_Needle_Expression_General}
\left\{
\begin{aligned}
\partial_t \F^{\omega_k,\tau_k}_t(x) & = \left( \D_x u^* \left( t , \Phi^{v,u^*}_{(\tau_k,t)}(x) \right) + \D_x v[\mu^*(t)] \left( t , \Phi^{v,u^*}_{(\tau_k,t)}(x) \right) \right) \F^{\omega_k,\tau_k}_t(x) \\
& + \INTDom{\BGamma^{v}_{\left( t ,\Phi^{v,u^*}_{(\tau_k,t)}(x) \right)} \left( \Phi^{v,u^*}_{(\tau_k,t)}(y) \right) \F^{\omega_k,\tau_k}_t(y)}{\R^d}{\mu^*(\tau_k)(y)}, \\
\F^{\omega_k,\tau_k}_{\tau_k}(x) & = \omega_k(x) - u^*(\tau_k,x). 
\end{aligned}
\right.
\end{equation}
\end{enumerate}
\end{lem}

\begin{proof}
The proof of this result is similar to that of  \cite[Lemma 5]{PMPWass}, with some extra technicalities arising from the induction argument performed on the non-local terms. By definition of a package of needle-like variations, the perturbed controls $\tilde{u}_e(\cdot,\cdot)$ generate well-defined flows of diffeomorphisms $(\Phi^{v,\tilde{u}_e}_{(0,t)}[\mu^0](\cdot))_{t \in [0,T]}$ which are such that 
\begin{equation*}
\tilde{\mu}_e(t) = \Phi^{v,\tilde{u}_e}_{(0,t)}[\mu^0] \circ \Phi^{v,u^*}_{(t,0)}[\mu^*(t)](\cdot)_{\#} \mu^*(t).
\end{equation*}
Hence, items (i), (ii) and (iii) hold for any $e \in [0,\epsilon_N]^N$ with $\G^N_t(e,\cdot) = \Phi^{v,\tilde{u}_e}_{(0,t)}[\mu^0] \circ \Phi^{v,u^*}_{(t,0)}[\mu^*(t)](\cdot)$.

We focus our attention on the proof by induction of (iv). Let $t \in [0,T]$ be such that $\iota(t) = 1$. By \eqref{eq:Needle_Variation}, one has that
\begin{equation*}
\begin{aligned}
\tilde{\mu}_e(t) & = \Phi^{v,u^*}_{(\tau_1,t)}[\tilde{\mu}_e(\tau_1)] \circ \Phi^{v,\omega_1}_{(\tau_1-e_1,\tau_1)} [\mu^*(\tau_1-e_1)] \circ \Phi^{v,u^*}_{(t,\tau_1-e_1)}[\mu^*(t)] (\cdot)_{\#} \mu^*(t) \\ 
& = \Phi^{v,u^*}_{(\tau_1,t)}[\tilde{\mu}_e(\tau_1)] \circ \Phi^{v,\omega_1}_{(\tau_1-e_1,\tau_1)} [\mu^*(\tau_1-e_1)] \circ \Phi^{v,u^*}_{(\tau_1,\tau_1-e_1)} [\mu^*(\tau_1)] \circ \Phi^{v,u^*}_{(t,\tau_1)}[\mu^*(t)] (\cdot)_{\#} \mu^*(t).
\end{aligned}
\end{equation*}
Invoking Lebesgue's Differentiation Theorem (see e.g. \cite[Corollary 2.23]{AmbrosioFuscoPallara}) along with the continuity of $e \mapsto v[\tilde{\mu}_{e}(t)](t,\cdot)$ in the $C^0$-norm topology, it holds that 
\begin{equation*}
\begin{aligned}
& \Phi^{v,\omega_1}_{(\tau_1-e_1,\tau_1)} [\mu^*(\tau_1-e_1)](x) \\
= ~  &  x + \INTSeg{ \left( v[\tilde{\mu}_e(t)] \left( t, \Phi^{v,\omega_1}_{(\tau_1-e_1,t)} [\mu^*(\tau_1-e_1)](x) \right) + \omega_1 \left( \Phi^{v,\omega_1}_{(\tau_1-e_1,t)} [\mu^*(\tau_1-e_1)](x) \right) \right)}{t}{\tau_1-e_1}{\tau_1} \\
= ~ & x + e_1 \Big( v[\mu^*(\tau_1)] \left(\tau_1,x) \right) + \omega_1 (x) \Big) + o(e_1)
\end{aligned}
\end{equation*}
as well as
\begin{equation*}
\begin{aligned}
\Phi^{v,u^*}_{(\tau_1,\tau_1-e_1)} [\mu^*(\tau_1)](x) & = x - \INTSeg{ \left( v[\tilde{\mu}_e(t)] \left( t, \Phi^{v,u^*}_{(\tau_1,t)} [\mu^*(\tau_1)](x) \right) + u^* \left( t , \Phi^{v,u^*}_{(\tau_1,t)} [\mu^*(\tau_1)](x) \right) \right)}{t}{\tau_1-e_1}{\tau_1} \\
& = x - e_1 \Big( v[\mu^*(\tau_1)] \left(\tau_1,x) \right) + u^*(\tau_1,x) \Big) + o(e_1).
\end{aligned}
\end{equation*}
Chaining these two expansions, we obtain that 
\begin{equation*}
\begin{aligned}
\tilde{\mu}_e(\tau_1) & = \Phi^{v,\omega_1}_{(\tau_1-e_1,\tau_1)} [\mu^*(\tau_1-e_1)] \circ \Phi^{v,u^*}_{(\tau_1,\tau_1-e_1)} [\mu^*(\tau_1)](\cdot)_{\#} \mu^*(\tau_1) \\
& = \big( \Id + e_1 [\omega_1(\cdot) - u^*(\tau_1,\cdot)] + o(e_1) \big)_{\#} \mu^*(\tau_1). 
\end{aligned}
\end{equation*}
We can now proceed to compute the induced first-order expansion on the non-local flows $\Phi^{v,u^*}_{(\tau_1,t)}[\tilde{\mu}_e(\tau_1)](\cdot)$ as follows
\begin{equation*}
\begin{aligned}
& \Phi^{v,u^*}_{(\tau_1,t)}[\tilde{\mu}_e(\tau_1)] \Big( x + e_1 \left[ \omega_1(x) - u^*(\tau_1,x) \right] + o(e_1) \Big) \\
= ~ & \Phi^{v,u^*}_{(\tau_1,t)}[\tilde{\mu}_e(\tau_1)](x) + e_1 \D_x \Phi^{v,u^*}_{(\tau_1,t)}[\tilde{\mu}_e(\tau_1)](x) \cdot \left[ \omega_1(x) - u^*(\tau_1,x) \right] + o(e_1) \\
= ~ & \Phi^{v,u^*}_{(\tau_1,t)}[\mu^*(\tau_1)](x) + e_1 \Big( \D_x \Phi^{v,u^*}_{(\tau_1,t)}[\mu^*(\tau_1)](x) \cdot \left[ \omega_1(x) - u^*(\tau_1,x) \right] + w_1(t,x) \Big) + o(e_1)
\end{aligned} 
\end{equation*}
where $w_1(\cdot,\cdot)$ is defined as in Proposition \ref{prop:Wasserstein_Differential_flow}, and where we used the fact that the $e \mapsto \D_x \Phi^{v,u^*}_{(\tau_1,\cdot)}[\tilde{\mu}_e(\cdot)](\cdot)$ is continuous as a consequence of hypothesis \textnormal{\textbf{(H1)}}-\textnormal{\textbf{(H2)}}. Introducing for all times $t \in [\tau_1,T]$ the map
\begin{equation*}
\F^{\omega_1,\tau_1}_t : x \in \overline{B(0,R_T)} \mapsto \D_x \Phi^{v,u^*}_{(\tau_1,t)}[\mu^*(\tau_1)](x) \left[ \omega_1(x) - u^*(\tau_1,x) \right] + w_1(t,x)
\end{equation*}
and invoking the statements of Proposition \ref{prop:Classical_Differential_flow} and Proposition \ref{prop:Wasserstein_Differential_flow}, we have that both \eqref{eq:TaylorExpansion_Needle_General} and \eqref{eq:FirstOrder_Needle_Expression_General} hold for any $e_1 \in [0,\bar{\epsilon}_N]$ and all times $t \in [0,T]$ such that $\iota(t) = 1$. 

Let us now assume that \eqref{eq:TaylorExpansion_Needle_General} and \eqref{eq:FirstOrder_Needle_Expression_General} hold for all times $t \in [0,T]$ such that $\iota(t) = k-1$, i.e.
\begin{equation}
\label{eq:General_Induction}
\tilde{\mu}_e(t) = \G_t^N(e,\cdot)_{\#} \mu^*(t) = \left( \Id + \sum\limits_{l=1}^{k-1} e_l \F^{\omega_l,\tau_k}_t \circ \Phi^{v,u^*}_{(t,\tau_l)}[\mu^*(t)](\cdot) + o(|e|) \right)_{\raisebox{0.4cm}{\scalebox{0.8}{\#}}} \mu^*(t),
\end{equation}
for $e \in [0,\bar{\epsilon}_N]^N$. By definition \eqref{eq:Needle_Variation} of an $N$-package of needle-like variations, we have that
\begin{equation*}
\tilde{\mu}_e(\tau_k) = \Phi^{v,\omega_k}_{(\tau_k-e_k,\tau_k)} [\tilde{\mu}_e(\tau_k-e_k)] \circ \Phi^{v,u^*}_{(\tau_k,\tau_k-e_k)}[\tilde{\mu}_e(\tau_k)] \circ \Phi^{v,u^*}_{(\tau_{k-1},\tau_k)}[\tilde{\mu}_e(\tau_{k-1})](\cdot)_{\#} \tilde{\mu}_e(\tau_{k-1}).
\end{equation*}
As in the initialization step, we can write using Lebesgue's Differentiation Theorem that
\begin{equation}
\label{eq:General_Expansion1}
\Phi^{v,\omega_k}_{(\tau_k-e_k,\tau_k)} [\tilde{\mu}_e(\tau_k-e_k)] \circ \Phi^{v,u^*}_{(\tau_k,\tau_k-e_k)}[\tilde{\mu}_e(\tau_k)](x) = x + e_k \left[ \omega_k(x) - u^*(\tau_k,x) \right] + o(e_k).
\end{equation}
Furthermore, invoking the induction hypothesis \eqref{eq:General_Induction} and the results of Proposition \ref{prop:Chainrule}, we obtain  that
\begin{equation}
\label{eq:General_Expansion2}
\begin{aligned}
& \Phi^{v,u^*}_{(\tau_{k-1},\tau_k)}[\tilde{\mu}_e(\tau_{k-1})] \left( x + \sum\limits_{l=1}^{k-1} \Big( e_l \F^{\omega_l,\tau_l}_{\tau_{k-1}} \circ \Phi^{v,u^*}_{(\tau_{k-1},\tau_l)}(x) + o(e_l) \Big) \right) \\
= ~ & \Phi^{v,u^*}_{(\tau_{k-1},\tau_k)}[\tilde{\mu}_e(\tau_{k-1})](x) + \sum_{l = 1}^{k-1} \Big( e_l \D_x \Phi^{v,u^*}_{(\tau_{k-1},\tau_k)}[\tilde{\mu}_e(\tau_{k-1})](x)  \F^{\omega_l,\tau_l}_{\tau_{k-1}} \circ \Phi^{v,u^*}_{(\tau_{k-1},\tau_l)}(x) + o(e_l) \Big) \\
= ~ & \Phi^{v,u^*}_{(\tau_{k-1},\tau_k)}[\mu^*(\tau_{k-1})](x) + \sum_{l = 1}^{k-1} e_l \left( \D_x \Phi^{v,u^*}_{(\tau_{k-1},\tau_k)}[\mu^*(\tau_{k-1})](x) \F^{\omega_l,\tau_l}_{\tau_{k-1}} \circ \Phi^{v,u^*}_{(\tau_{k-1},\tau_l)}(x) + w_l(\tau_k,x) + o(e_l) \right) 
\end{aligned}
\end{equation} 
where the maps $(w_l(\cdot,\cdot))_{1 \leq l \leq k-1}$ are defined as in Proposition \ref{prop:Wasserstein_Differential_flow} with $\F_l(\cdot) \equiv \F^{\omega_l,\tau_l}_{\tau_{k-1}} \circ \Phi^{u^*}_{(\tau_{k-1},\tau_l)}[\mu^*(\tau_{k-1})](\cdot)$. Plugging together equation \eqref{eq:Differential_LocalFlow} of Proposition \ref{prop:Classical_Differential_flow} and equation \eqref{eq:Differential_NonLocalFlow} of Proposition \ref{prop:Wasserstein_Differential_flow}, one can see that the maps
\begin{equation*}
t \in [\tau_{k-1},\tau_k] \mapsto \D_x \Phi^{v,u^*}_{(\tau_{k-1},t)}[\mu^*(\tau_{k-1})] \left( \Phi^{v,u^*}_{(\tau_l,\tau_{k-1})}[\mu^*(\tau_l)](x) \right) \F^{\omega_l,\tau_l}_{\tau_{k-1}}(x) + w_l \left( t,\Phi^{v,u^*}_{(\tau_l,\tau_{k-1})}[\mu^*(\tau_l)](x) \right)
\end{equation*}
are solutions of \eqref{eq:FirstOrder_Needle_Expression_General} on $[\tau_{k-1},\tau_k]$ with initial condition $\F^{\omega_l,\tau_l}_{\tau_{k-1}}(\cdot)$ at time $\tau_{k-1}$ for any $l \in \{1,\dots,k-1\}$. By Cauchy-Lipschitz uniqueness, we can therefore extend the definition of the maps $t \mapsto \F^{\omega_l,\tau_l}_{t}(x)$ to the whole of $[\tau_l,\tau_k]$ for any $l \in \{1,\dots,k-1\}$. 

Chaining the expansions \eqref{eq:General_Expansion1} and \eqref{eq:General_Expansion2} along with our previous extension argument, we obtain that both \eqref{eq:TaylorExpansion_Needle_General} and \eqref{eq:FirstOrder_Needle_Expression_General} hold up to time $\tau_k$, i.e.
\begin{equation*}
\G_{\tau_k}^N(e,\cdot) = \Id + \sum\limits_{l=1}^{k} e_l \F^{\omega_l,\tau_l}_{\tau_k} \circ \Phi^{u^*}_{(\tau_k,\tau_l)}[\mu^*(\tau_k)](\cdot) + o(|e|)
\end{equation*}
for any $e \in [0,\bar{\epsilon}_N]^N$. Performing yet another coupled Taylor expansion of the same form on the expression
\begin{equation*}
\tilde{\mu}_e(t) = \Phi^{v,u^*}_{(\tau_k,t)}[\tilde{\mu}_e(\tau_k)](\cdot)_{\#} \tilde{\mu}_e(\tau_k), 
\end{equation*}
and invoking the same extension argument yields the full induction step for all times $t \in [0,T]$ such that $\iota(t) = k$. Hence, we have proven that item (iv) holds for all $e \in [0,\bar{\epsilon}_N]^N$. Using Lemma \ref{lem:Continuous_Extension}, we can now extend the map $e \in [0,\bar{\epsilon}_N]^N \mapsto \G_t^N(e,\cdot) \in C^0(\overline{B(0,R_T)},\R^d)$ to the whole of $[-\bar{\epsilon}_N,\bar{\epsilon}_N]^N$ in a continuous and bounded way, uniformly with respect to $t \in [0,T]$, while preserving its differential at $e = 0$. 
\end{proof}

In the sequel, we drop the explicit dependence of the flow maps on their starting measures and adopt the simplified notation $\Phi^{v,u^*}_{(s,t)}(x) \equiv \Phi^{v,u^*}_{(s,t)}[\mu(s)](x)$ for clarity and conciseness.


\bigskip

\begin{flushleft}
\underline{\textbf{Step 2 : First-order optimality condition}}
\end{flushleft}

\bigskip

In Lemma \ref{lem:First_order_needle_General}, we derived the analytical expression of the first-order perturbation induced by a $N$-package of needle-like variations on the solution of a controlled non-local continuity equation. By the very definition of an $N$-package of needle-like variations, we know that the finite-dimensional optimization problem
\begin{equation*}
(\Ppazo_N) ~~ \left\{
\begin{aligned}
& \min\limits_{e \in [0,\bar{\epsilon}_N]^N} \left[ \INTSeg{L(t,\tilde{\mu}_e(t),\tilde{u}_e(t))}{t}{0}{T} + \varphi(\tilde{\mu}_e(T)) \right] \\
& \text{s.t.} \left\{
\begin{aligned}
& \Psi^E(\tilde{\mu}_e(T)) = 0, ~~ \Psi^I(\tilde{\mu}_e(T)) & \leq 0, \\
&\max_{t \in [0,T]} \Lambda(t,\tilde{\mu}_e(t)) \leq 0, 
\end{aligned}
\right. 
\end{aligned}
\right.
\end{equation*}
admits $e=0$ as an optimal solution in $[0,\bar{\epsilon}_N]^N$ for $\bar{\epsilon}_N$ small enough.

In the following lemma, we check that the functionals involved in $(\Ppazo_N)$ meet the requirements of the Lagrange multiplier rule stated in Theorem \ref{thm:Lagrange_MP}. We also compute their first-order variation induced by the package of needle-like variations at $e = 0$. 

\begin{lem}[Differentiability and calmness of the functionals involved in $(\Ppazo_N)$]
\label{lem:Derivative_RCSC} The maps $e \in [-\bar{\epsilon}_N,\bar{\epsilon}_N]^N \mapsto \varphi(\tilde{\mu}_e(T)),~ \Psi^E(\tilde{\mu}_e(T)),~ \Psi^I(\tilde{\mu}_e(T))$ and $e \in [-\bar{\epsilon}_N,\bar{\epsilon}_N]^N \mapsto \INTSeg{L(t,\tilde{\mu}_e(t),\tilde{u}_e(t))}{t}{0}{T}$ are calm and Fr\'echet- differentiable at $e=0$. Their Fr\'echet derivative in a direction $\sigma \in [0,\bar{\epsilon}_N]^N$ are respectively given by
\begin{equation}
\label{eq:Derivative_FC}
\left\{
\begin{aligned}
\nabla_e \big( \varphi(\tilde{\mu}_e(T)) \big)_{\vert e=0} (\sigma) & = \sum_{k=1}^N \sigma_k \INTDom{ \left\langle \nabla_{\mu} \varphi(\mu^*(T))(x) , \F^{\omega_k,\tau_k}_T \circ \Phi^{v,u^*}_{(T,\tau_k)}(x) \right\rangle}{\R^d}{\mu^*(T)(x)}, \\
\nabla_e \big( \Psi^I_i(\tilde{\mu}_e(T)) \big)_{\vert e=0} (\sigma) & = \sum_{k=1}^N \sigma_k \INTDom{ \left\langle \nabla_{\mu} \Psi^I_i(\mu^*(T))(x) , \F^{\omega_k,\tau_k}_T \circ \Phi^{v,u^*}_{(T,\tau_k)}(x) \right\rangle}{\R^d}{\mu^*(T)(x)}, \\
\nabla_e \big( \Psi^E_j(\tilde{\mu}_e(T)) \big)_{\vert e=0} (\sigma) & = \sum_{k=1}^N \sigma_k \INTDom{ \left\langle \nabla_{\mu}  \Psi^E_j(\mu^*(T))(x) , \F^{\omega_k,\tau_k}_T \circ \Phi^{v,u^*}_{(T,\tau_k)}(x) \right\rangle}{\R^d}{\mu^*(T)(x)}, \\
\end{aligned} 
\right.
\end{equation}
for any $(i,j) \in \{1,\dots,n\} \times \{1,\dots,m\}$ and
\begin{equation}
\label{eq:Derivative_RunningCost}
\begin{aligned}
\nabla_e \left( \INTSeg{L(t,\tilde{\mu}_e(t),\tilde{u}_e(t))}{t}{0}{T} \right)_{\vert e=0} (\sigma) = & \sum_{k=1}^N \sigma_k \INTSeg{\INTDom{ \left\langle \nabla_{\mu} L(t,\mu^*(t),u^*(t))(x) , \F^{\omega_k,\tau_k}_t \circ \Phi^{v,u^*}_{(t,\tau_k)}(x) \right\rangle}{\R^d}{\mu^*(t)(x)}}{t}{\tau_k}{T} \\
 + & \sum_{k=1}^N \sigma_k \Big( L(\tau_k,\mu^*(\tau_k),u^*(\tau_k)) - L(\tau_k,\mu^*(\tau_k),\omega_k) \Big).
\end{aligned}
\end{equation}

The maps $e \in [-\bar{\epsilon}_N,\bar{\epsilon}_N]^N \mapsto \max_{t \in [0,T]} \Lambda_1(t,\tilde{\mu}_e(t)),\dots,~\max_{t \in [0,T]} \Lambda_r(t,\tilde{\mu}_e(t))$ are calm and locally Lipschitz around $e=0$. Their Michel-Penot derivatives in a direction $\sigma \in [0,\bar{\epsilon}_N]^N$ are given by
\begin{equation}
\label{eq:Derivative_StateConstraint}
d_{\MP} \Big( \max_{t \in [0,T]} \Lambda_l(t,\tilde{\mu}_e(t))] \Big)(0 \, ;\sigma) = \sum_{k=1}^N \sigma_k \INTSeg{\INTDom{ \left\langle \nabla_{\mu} \Lambda_l(t,\mu^*(t))(x) , \F^{\omega_k,\tau_k}_t \circ \Phi^{v,u^*}_{(t,\tau_k)}(x) \right\rangle}{\R^d}{\mu^*(t)}}{\varpi^N_l(t)}{0}{T},
\end{equation}
where the Borel measures $\varpi^N_l \in \M_+([0,T])$ are such that $\supp(\varpi_l^N) = \big\{ t \in [0,T] ~\text{s.t.}~ \Lambda_l(t,\mu^*(t)) = 0 \big\}$ and $\Norm{\varpi_l ^N}_{TV} = 1$ for any $l \in \{1,\dots,r\}$. 
\end{lem}

\begin{proof}
The calmness property of the maps $e \mapsto \varphi(\tilde{\mu}_e(T)),~ \Psi^E(\tilde{\mu}_e(T)),~ \Psi^I(\tilde{\mu}_T(e))$ and $e \mapsto \INTSeg{L(t,\tilde{\mu}_e(t),\tilde{u}_e(t))}{t}{0}{T}$ at $e=0$ stems from the fact that they are compositions of Fr\'echet-differentiable and locally Lipschitz mappings as a by-product of hypotheses \textbf{(H4)} and \textbf{(H5)} and Lemma \ref{lem:First_order_needle_General}. The differentials of the final cost and boundary constraints can be computed with a direct application of Proposition \ref{prop:Chainrule}.

We split the computation of the first-order variation at $e=0$ of the running cost functional into two parts. One can first derive that
\begin{equation*}
\begin{aligned}
\INTSeg{\Big( L(t,\tilde{\mu}_e(t),\tilde{u}_e(t)) - L(t,\tilde{\mu}_e(t),u^*(t)) \Big)}{t}{0}{T} & = \sum_{k=1}^N \INTSeg{\Big( L(t,\tilde{\mu}_e(t),\omega_k) - L(t,\tilde{\mu}_e(t),u^*(t)) \Big)}{t}{\tau_k-e_k}{\tau_k} \\
& = \sum_{k=1}^N e_k \Big( L(\tau_k,\tilde{\mu}_e(\tau_k),\omega_k) - L(\tau_k,\tilde{\mu}_e(\tau_k),u^*(\tau_k)) \Big) + o(|e|) \\
& = \sum_{k=1}^N e_k \Big( L(\tau_k,\mu^*(\tau_k),\omega_k) - L(\tau_k,\mu^*(\tau_k),u^*(\tau_k)) \Big) + o(|e|), \\
\end{aligned}
\end{equation*}  
by Lebesgue's Differentiation Theorem, since $\tau_k$ is a Lebesgue points of $t \in [0,T] \mapsto L(t,\mu^*(t),\cdot) \in C^0(U,\R)$ and the maps $e \mapsto L(\tau_k,\tilde{\mu}_e(\tau_k),u^*(\tau_k))$ and $e \mapsto L(\tau_k,\tilde{\mu}_e(\tau_k),\omega_k)$ are continuous for all $k \in \{1,\dots,N\}$. Furthermore, invoking the Wasserstein chainrule of Proposition \ref{prop:Chainrule} along with the results of Lemma \ref{lem:First_order_needle_General}, we have that 
\begin{equation*}
L(t,\tilde{\mu}_e(t),u^*(t)) - L(t,\mu^*(t),u^*(t)) = \sum_{k=1}^N e_k \INTDom{\left\langle \nabla_{\mu} L(t,\mu^*(t),u^*(t))(x) , \F^{\omega_k,\tau_k}_t \circ \Phi^{v,u^*}_{(t,\tau_k)}(x) \right\rangle}{\R^d }{\mu^*(t)(x)} + o(|e|).
\end{equation*}
for $\Lcal^1$-almost every $t \in [0,T]$. Combining these expansions with an application of Lebesgue's Dominated Convergence Theorem yields \eqref{eq:Derivative_RunningCost}. 

We now turn our attention to the state constraints functionals. By hypothesis \textbf{(H6)} and Proposition \ref{prop:Chainrule}, the maps $e \mapsto \Lambda_l(t,\tilde{\mu}_e(t))$ are Fr\'echet-differentiable at $e=0$ for any $l \in \{1,\dots,r\}$, uniformly with respect to $t \in [0,T]$. Moreover, the functional 
\begin{equation*}
\gamma \in C^0([0,T],\R) \mapsto \max_{t \in [0,T]} \gamma(t)
\end{equation*}
is locally convex and therefore locally Lipschitz over $C^0([0,T],\R)$. Hence, the maps $e \mapsto \max_{t \in [0,T]} \Lambda_l(t,\tilde{\mu}_e(t))$ are calm at $e=0$ as compositions of Fr\'echet-differentiable and locally Lipschitz mappings. By Proposition \ref{prop:Prop_MPsubdiff}-(c), we can compute their Michel-Penot derivatives in a direction $\sigma \in [0,\bar{\epsilon}_N]^N$ as follows
\begin{equation*}
d_{\MP} \left( \max_{t \in [0,T]} \Lambda_l(t,\tilde{\mu}_e(t)) \right)(0 \, ;\sigma) = \INTSeg{\nabla_e \Big( \Lambda_l(t,\tilde{\mu}_e(t)) \Big)_{\vert e=0}(\sigma)}{\varpi_l^N(t)}{0}{T}
\end{equation*}
where $\varpi_l^N \in \M_+([0,T])$ belongs to the convex subdifferential of the $C^0([0,T],\R)$-norm evaluated at $\Lambda_l(\cdot,\mu^*(\cdot))$. This subdifferential can be classically characterized (see e.g. \cite[Section 4.5.1]{IoffeTihomirov}) as the set of Borel regular measures such that
\begin{equation*}
\Norm{\varpi_l^N}_{TV} = 1 \hspace{0.2cm} \text{and} \hspace{0.2cm} \supp(\varpi_l^N) = \Big\{ t \in [0,T] ~\text{s.t.}~ \Lambda_l(t,\mu^*(t)) = 0 \Big\}.
\end{equation*}
Invoking again Proposition \ref{prop:Chainrule} and Lemma \ref{lem:First_order_needle_General}, we can write the differential of $e \mapsto \Lambda(t,\tilde{\mu}_e(t))$ at $e=0$ evaluated in a direction $\sigma \in [0,\bar{\epsilon}_N]^N$ as 
\begin{equation*}
\nabla_e \Big( \Lambda_l(t,\tilde{\mu}_e(t)) \Big)_{\vert e=0}(\sigma) = \sum_{k=1}^N \sigma_k \INTDom{\left\langle \nabla_{\mu} \Lambda_l(t,\mu^*(t))(x) , \F^{\omega_k,\tau_k}_t \circ \Phi^{v,u^*}_{(t,\tau_k)}(x) \right\rangle}{\R^d}{\mu^*(t)(x)}
\end{equation*}
for all times $t \in [0,T]$. Using the measure-theoretic version of Lebesgue's Dominated Convergence Theorem (see e.g. \cite[Theorem 1.21]{AmbrosioFuscoPallara}), we conclude that \eqref{eq:Derivative_StateConstraint} holds as well, which ends the proof of Lemma \ref{lem:Derivative_RCSC}. 
\end{proof}

Using the results of Lemma \ref{lem:Derivative_RCSC}, we can apply the Lagrange multiplier rule of Theorem \ref{thm:Lagrange_MP} to $(\Ppazo_N)$ and obtain the existence of scalar multipliers $(\lambda_0^N,\lambda_1^N,\dots,\lambda_n^N,\eta_1^N,\dots,\eta_m^N,\theta_1^N,\dots\theta_r^N) \in \{0,1\} \times \R^n_+ \times \R^m \times \R_+^r$ such that 
\begin{equation*}
\left\{
\begin{aligned}
& 0 \in \partial_{\MP} \Bigg( \lambda_0^N \varphi(\tilde{\mu}_e(T)) + \lambda_0^N \INTSeg{L(t,\tilde{\mu}_e(t),\tilde{u}_e(t)) }{t}{0}{T} + \sum_{l=1}^r \theta_l^N \max_{t \in [0,T]} \Lambda_l(t,\tilde{\mu}_e(t)), ~~ & \textnormal{(S)} \\
& \hspace{4.725cm} + \sum\limits_{i = 1}^n \lambda_i^N \Psi^I_i(\tilde{\mu}_e(T)) + \sum\limits_{j = 1}^m \eta_j^N \Psi^E_j(\tilde{\mu}_e(T)) \Bigg)_{\vert e=0} + \Ncal([0,\bar{\epsilon}_N]^N,0) \\
& \lambda_0^N + \sum\limits_{i = 1}^n \lambda_i^N + \sum\limits_{j = 1}^m |\eta_j^N| + \sum\limits_{l = 1}^r \theta_l^N = 1,  ~~  & \textnormal{(NT)} \\
& \theta_l^N \max_{t \in [0,T]} \Lambda_l(t,\mu^*(t)) = 0 ~~ \text{for all $l \in \{1,\dots,r\}$,} ~~ & \textnormal{(CS)} \\
& \hspace{1.05cm} \lambda_i^N \Psi^I_i(\mu^*(T)) = 0 ~~ \text{for all $i \in \{ 1,\dots,n\}$.} \\
\end{aligned}
\right.
\end{equation*}

Since all the functions involved in the subdifferential inclusion (S) are calm, we can use the summation rule of Proposition \ref{prop:Prop_MPsubdiff}-(b) along with the characterization of MP-subdifferentials for Fr\'echet-differentiable functionals stated in Proposition \ref{prop:Prop_MPsubdiff}-(a) to obtain that 
\begin{equation*}
\begin{aligned}
- \nabla_e \left( \lambda^0_N \varphi(\tilde{\mu}_e(T))  + \lambda_0^N \INTSeg{L(t,\tilde{\mu}_e(t),\tilde{u}_e(t)) }{t}{0}{T} + \sum\limits_{i = 1}^n \lambda_i^N \Psi^I_i(\tilde{\mu}_e(T)) + \sum\limits_{j = 1}^m \eta_j^N \Psi^E_j(\tilde{\mu}_e(T))  \right)_{\vert e=0} & \\
\in \sum_{l=1}^r \theta^N_l \partial_{\MP} \left( \max_{t \in [0,T]} \Lambda_l(t,\tilde{\mu}_e(t)) \right)_{|e=0}&  + \Ncal([0,\bar{\epsilon}_N]^N,0). 
\end{aligned}
\end{equation*}
By combining the expressions of the gradients \eqref{eq:Derivative_FC}, \eqref{eq:Derivative_RunningCost} and the MP-derivative \eqref{eq:Derivative_StateConstraint} derived in Lemma \ref{lem:Derivative_RCSC},  along with the composition rule of Proposition \ref{prop:Prop_MPsubdiff}-(c) for MP-subdifferentials, we obtain that
\begin{equation*}
\begin{aligned}
& \sum_{k=1}^N \sigma_k \INTDom{\langle -\nabla_{\mu} \Scal_N(\mu^*(T))(x) , \F^{\omega_k,\tau_k}_T \circ \Phi^{v,u^*}_{(T,\tau_k)}(x) \rangle}{\R^d}{\mu^*(T)(x)} \\
- & \sum_{k=1}^N \sigma_k  \lambda_0^N \Big( L(\tau_k,\mu^*(\tau_k),u^*(\tau_k)) - L(\tau_k,\mu^*(\tau_k),\omega_k) \Big) \\
- & \sum_{k=1}^N \sigma_k  \INTSeg{\INTDom{ \left\langle \lambda_0^N \nabla_{\mu} L(t,\mu^*(t),u^*(t))(x) , \F^{\omega_k,\tau_k}_t \circ \Phi^{v,u^*}_{(t,\tau_k)}(x) \right\rangle}{\R^d}{\mu^*(t)(x)}}{t}{\tau_k}{T} \\
- & \sum_{k=1}^N \sigma_k \sum_{l=1}^r \theta_l^N \INTSeg{\INTDom{\left\langle \nabla_{\mu} \Lambda_l(t,\mu^*(t))(x) , \F^{\omega_k,\tau_k}_t \circ \Phi^{v,u^*}_{(t,\tau_k)}(x) \right\rangle}{\R^d}{\mu^*(t)(x)}}{\varpi_l^N(t)}{\tau_k}{T} \leq 0 
\end{aligned}
\end{equation*}
for any direction $\sigma \in [0,\bar{\epsilon}_N]^N$, where $\nabla_{\mu} \Scal_N(\mu^*(T))(\cdot)$ is defined as in \eqref{eq:FinalGradient_General}. By choosing particular vectors $\sigma \in [0,\bar{\epsilon}_N]^N$ which have all their components except one equal to $0$, this family of inequalities can be rewritten as 
\begin{equation}
\label{eq:OptimalityConditions_General}
\begin{aligned}
& \INTDom{\langle -\nabla_{\mu} \Scal_N(\mu^*(T))(x) , \F^{\omega_k,\tau_k}_T \circ \Phi^{v,u^*}_{(T,\tau_k)}(x) \rangle}{\R^d}{\mu^*(T)(x)} - \lambda_0^N \Big( L(\tau_k,\mu^*(\tau_k),u^*(\tau_k)) - L(\tau_k,\mu^*(\tau_k),\omega_k) \Big) \\
- & \INTSeg{\INTDom{ \left\langle \lambda_0^N \nabla_{\mu} L(t,\mu^*(t),u^*(t))(x) , \F^{\omega_k,\tau_k}_t \circ \Phi^{v,u^*}_{(t,\tau_k)}(x) \right\rangle}{\R^d}{\mu^*(t)(x)}}{t}{\tau_k}{T} \\
- & \sum_{l=1}^r \INTSeg{\INTDom{\left\langle \nabla_{\mu} \Lambda_l(t,\mu^*(t))(x) , \F^{\omega_k,\tau_k}_t \circ \Phi^{v,u^*}_{(t,\tau_k)}(x) \right\rangle}{\R^d}{\mu^*(t)(x)}}{\varpi_l^N(t)}{\tau_k}{T} \leq 0
\end{aligned}
\end{equation}
for all $k \in \{1,\dots,N\}$, where we redefined the notation $\varpi_l^N \equiv \theta_l^N \varpi_l^N$.


\bigskip

\begin{flushleft}
\underline{\textbf{Step 3 : Backward dynamics and partial Pontryagin maximization condition}}
\end{flushleft}

\bigskip

The next step of our proof is to introduce a suitable notion of state-costate variable transporting the family of inequalities \eqref{eq:OptimalityConditions_General} derived at time $T$ to the base points $(\tau_1,\dots,\tau_N)$ of the needle-like variations while generating a Hamiltonian dynamical structure. To this end, we build for all $N \geq 1$ a curve $\nu^*_N(\cdot) \in \Lip([0,T],\Pcal_c(\R^{2d}))$ solution of the forward-backward system of continuity equations
\begin{equation}
\label{eq:HamiltonianFlow_General}
\left\{
\begin{aligned}
& \partial_t \nu_N^*(t) + \nabla \cdot (\V_N^*[\nu_N^*(t)](t,\cdot,\cdot)\nu_N^*(t)) = 0 ~~ & \text{in $[0,T] \times \R^{2d}$}, \\
& \pi^1_{\#} \nu_N^*(t) = \mu^*(t) ~~ & \text{for all $t \in [0,T]$}, \\
& \nu^*_N(T) = \big( \Id \times (-\nabla_{\mu} \Scal_N(\mu^*(T))) \big)_{\#} \mu^*(T).
\end{aligned}
\right.
\end{equation} 
Here, the non-local velocity field $\V^*_N[\cdot](\cdot,\cdot,\cdot)$ is given for $\Lcal^1$-almost every $t \in [0,T]$ any $(x,r,\nu) \in \R^{2d} \times \Pcal_c(\R^{2d})$ by 
\begin{equation*}
\V^*_N[\nu](t,x,r) = \begin{pmatrix}
v[\pi^1_{\#} \nu](t,x) + u^*(t,x) \\ \\
\lambda_0^N \nabla_{\mu} L(t,\pi^1_{\#}\nu,u^*(t))(x) + \nabla_{\mu} \Ccal(t,\pi^1_{\#}\nu,\zeta^*_N(t),u^*(t))(x) \\ - \BGamma^v[\nu](t,x) - \D_x v[\pi^1_{\#} \nu](t,x)^{\top} r - \D_x u^*(t,x)^{\top} r
\end{pmatrix},
\end{equation*}
where we introduced the notation
\begin{equation}
\label{eq:BGammaBis_def}
\BGamma^v[\nu](t,x) = \INTDom{\Big( \BGamma^v_{(t,y)}(x) \Big)^{\top} p \, }{\R^{2d}}{\nu(y,p)}.
\end{equation}

Notice that the transport equation \eqref{eq:HamiltonianFlow_General} does not satisfy the classical hypotheses of Theorem \ref{thm:Nonlocal_PDE}. Following a methodology introduced in our previous work \cite{PMPWass}, it is possible to circumvent this difficulty by building explicitly a solution of \eqref{eq:HamiltonianFlow_General} relying on the \textit{cascade structure} of the equations. 

\begin{lem}[Definition and well-posedness of solutions of \eqref{eq:HamiltonianFlow_General}]
\label{lem:Wellposedness_General} Let $(u^*(\cdot),\mu^*(\cdot))$ be an optimal pair control-trajectory for $(\Ppazo)$. For $\mu^*(T)$-almost every $x \in \R^d$, we consider the family of backward flows $(\Psi^{x,N}_{(T,t)}(\cdot))_{t \leq T}$ solution of the non-local Cauchy problems
\begin{equation}
\label{eq:Marginal_TransportEq_General}
\left\{
\begin{aligned}
\partial_t w_x(t,r) & =  \lambda_0^N \nabla_{\mu} L(t,\mu^*(t),u^*(t)) \left( \Phi_{(T,t)}^{v,u^*}(x) \right) + \nabla_{\mu} \Ccal(t,\mu^*(t),\zeta_N^*(t),u^*(t)) \left( \Phi_{(T,t)}^{v,u^*}(x) \right) \\
& - \D_x u^* \left( t , \Phi_{(T,t)}^{v,u^*}(x) \right)^{\top} w_x(t,r) - \D_x v[\mu^*(t)] \left( t , \Phi_{(T,t)}^{v,u^*}(x) \right)^{\top} w_x(t,r) \\
& - \INTDom{\BGamma^v_{ \left( t , \Phi^{v,u^*}_{(T,t)}(y) \right) } \left( \Phi^{v,u^*}_{(T,t)}(x) \right)^{\top} w_y(t,p)}{\R^{2d}}{ \big( (\Id \times (-\nabla_{\mu} \Scal_N(\mu^*(T))))_{\#} \mu^*(T)  \big)(y,p)} \\
w_x(T,r) & = r,
\end{aligned}
\right.
\end{equation}
and the associated curves of measures 
\begin{equation*}
\sigma^*_{x,N} : t \mapsto \Psi^{x,N}_{(T,t)}(\cdot)_{\#} \delta_{(-\nabla_{\mu}\Scal_N(\mu^*(T))(x))}.
\end{equation*}
Define the map $\nu^*_{T,N} : t \in [0,T] \mapsto \int \sigma^*_{x,N}(t) \textnormal{d} \mu^*(T)(x) \in \Pcal_c(\R^{2d})$. Then, the curve $\nu_N^* : t \in [0,T] \mapsto (\Phi^{v,u^*}_{(T,t)}(\cdot) , \Id)_{\#} \nu^*_{T,N}(t)$ is the unique solution of \eqref{eq:HamiltonianFlow_General}. Moreover, there exists two constants $R_T',L_T' > 0$ such that
\begin{equation*}
\supp(\nu_N^*(t)) \subset \overline{B_{2d}(0,R_T')} \qquad \text{and} \qquad W_1(\nu_N^*(t),\nu_N^*(s)) \leq L_T' |t-s|, 
\end{equation*}  
for all $s,t \in [0,T]$.
\end{lem}

\begin{proof}
Let us denote by $\Omega \subset \R^{2d}$ a compact set such that 
\begin{equation*}
\bigcup_{N \geq 1} \supp \Big( \big(\Id \times (-\nabla_{\mu}\Scal_N(\mu^*(T))) \big)_{\#} \mu^*(T) \Big) \subseteq \Omega.
\end{equation*}
Such a set exists since the maps $(\nabla_{\mu} \Scal_N(\mu^*(T))(\cdot))$ are continuous by \textbf{(H4)}, as well as uniformly bounded as a consequence of the non-triviality condition (NT) on the Lagrange multipliers $(\lambda_0^N,\dots,\lambda_n^N,\eta_1^N,\dots,\eta_m^N)$.

The existence and uniqueness of the maps $(t,x,r) \mapsto w_x(t,r)$ solving the family of non-local Cauchy problems \eqref{eq:Marginal_TransportEq_General} can be obtained under hypotheses \textnormal{\textbf{(H)}}, as a consequence of Banach fixed point Theorem in the spirit of \cite[Proposition 5]{PMPWass}. In this context, the Banach space under consideration is that of all maps $w : [0,T] \times \Omega \rightarrow \R^d$ endowed with the norm
\begin{equation*}
\Norm{w_{\cdot}(\cdot,\cdot)}_* = \inf \left\{ M > 0 ~\text{s.t.}~ \NormC{w_x(\cdot,\cdot)}{0}{[0,T] \times \R^d,\R^d} \leq M ~~ \text{for $\mu^*(T)$-almost every $x \in \R^d$} \right\}.
\end{equation*}
By an application of Gr\"onwall's Lemma to \eqref{eq:Marginal_TransportEq_General}, it holds that $(t,r) \in [0,T] \times \pi^2(\Omega) \mapsto \Psi^{x,N}_{(T,t)}(r)$ is bounded by a positive constant, uniformly with respect to $x \in \supp(\mu^*(T))$ and $N \geq 1$. This follows in particular from the uniform boundedness of the sequences of multipliers $(\lambda^0_N)$ and $(\zeta^*_N(\cdot))$. Therefore, there exists a uniform constant $R_T'>0$ such that 
\begin{equation*}
\supp(\nu^*_N(t)) \subset \overline{B_{2d}(0,R_T')}, 
\end{equation*}
for all $t \in [0,T]$. This in turn implies that the right-hand side of \eqref{eq:Marginal_TransportEq_General} is uniformly bounded, so that the maps $t \in [0,T] \mapsto \Psi^{x,N}_{(T,t)}(r)$ are Lispchitz, uniformly with respect to $(x,r) \in \Omega$ and $N \geq 1$. By applying again Gr\"owall's Lemma to the difference $|\Psi^{x,N}_{(T,t)}(r_2)-\Psi^{x,N}_{(T,t)}(r_1)|$ with $r_1,r_2 \in \pi^2(\Omega)$, we further obtain that $(t,r) \in [0,T] \times \pi^2(\Omega) \mapsto \Psi^{x,N}_{(T,t)}(r)$ is also Lipschitz regular, uniformly with respect to $x \in \supp(\mu^*(T))$ and $N \geq 1$. It can be checked by leveraging Kantorovich duality in the spirit of \cite[Lemma 6]{PMPWass} that this in turn yields the Lipschitz regularity of $t \in [0,T] \mapsto \sigma^*_{x,N}(t)$ uniformly with respect to $x \in \supp(\mu^*(T))$ and $N \geq 1$. An application of Proposition \ref{prop:Estimation_Barycenter} combined with the uniform Lipschitz regularity of $(t,x) \in [0,T] \times \pi^1(\Omega) \mapsto \Phi^{v,u^*}_{(T,t)}(x)$ to $\nu^*_N(\cdot)$ provides the existence of a uniform constant $L_T' > 0$ such that
\begin{equation*}
W_1(\nu^*_N(t),\nu^*_N(s)) \leq L_T' |t-s|, 
\end{equation*}
for any $s,t \in [0,T]$.

In order to prove that $\nu^*_N(\cdot)$ is indeed a solution of \eqref{eq:HamiltonianFlow_General}, take $\xi \in C^{\infty}_c(\R^{2d})$ and compute the time derivative
\begin{equation*}
\begin{aligned}
& \derv{}{t}{} \INTDom{\xi(x,r)}{\R^{2d}}{\nu_N^*(t)(x,r)} = \derv{}{t}{} \INTDom{ \INTDom{\xi \left( \Phi^{v,u^*}_{(T,t)}(x) ,r \right)}{\R^d}{\sigma^*_{x,N}(t)(r)}}{\R^d}{\mu^*(T)(x)} \\
= ~ & \INTDom{\INTDom{\left\langle \nabla_x \xi \left( \Phi^{v,u^*}_{(T,t)}(x) ,r \right) , v[\mu^*(t)] \left( t , \Phi^{v,u^*}_{(T,t)}(x) \right) + u^* \left( t , \Phi^{v,u^*}_{(T,t)}(x) \right) \right\rangle}{\R^d}{\sigma^*_{x,N}(t)(r)}}{\R^d}{\mu^*(T)(x)} \\
+ ~ & \INTDom{\INTDom{\left\langle \nabla_r \xi \left( \Phi^{v,u^*}_{(T,t)}(x) ,r\right) ,  \lambda_0^N \nabla_{\mu} L(t,\mu^*(t),u^*(t)) \left( \Phi^{v,u^*}_{(T,t)}(x)\right) \right\rangle}{\R^d}{\sigma^*_{x,N}(t)(r)}}{\R^d}{\mu^*(T)(x)} \\
+ ~ & \INTDom{\INTDom{\left\langle \nabla_r \xi \left( \Phi^{v,u^*}_{(T,t)}(x),r \right) , \nabla_{\mu} \Ccal(t,\mu^*(t),\zeta_N^*(t),u^*(t)) \left( \Phi^{v,u^*}_{(T,t)}(x) \right) \right\rangle}{\R^d}{\sigma^*_{x,N}(t)(r)}}{\R^d}{\mu^*(T)(x)} \\
- ~ & \INTDom{\INTDom{\Big\langle \nabla_r \xi \left( \Phi^{v,u^*}_{(T,t)}(x) ,r\right) , \D_x v[\mu^*(t)]\left( t , \Phi^{v,u^*}_{(T,t)}(x) \right)^{\top} r + \D_x u^* \left( t , \Phi^{v,u^*}_{(T,t)}(x) \right)^{\top} r \Big\rangle}{\R^d}{\sigma^*_{x,N}(t)(r)}}{\R^d}{\mu^*(T)(x)}
\\
- ~ & \INTDom{\INTDom{\Big\langle \nabla_r \xi \left( \Phi^{v,u^*}_{(T,t)}(x) ,r \right) , \BGamma^v[\nu_N^*(t)] \left( t , \Phi^{v,u^*}_{(T,t)}(x) \right) \Big\rangle}{\R^d}{\sigma^*_{x,N}(t)(r)}}{\R^d}{\mu^*(T)(x)}
\end{aligned}
\end{equation*}
where we used the fact that by Fubini's Theorem
\begin{equation*}
\begin{aligned}
\BGamma^v [\nu^*_N(t)] \left( t , \Phi^{v,u^*}_{(T,t)}(x) \right) & = \INTDom{\BGamma^v_{\left( t , y \right)} \left( \Phi^{v,u^*}_{(T,t)}(x) \right)^{\top} p \,}{\R^{2d}}{\nu^*_N(t)(y,p)} \\
& = \INTDom{\BGamma^v_{\left( t , \Phi^{v,u^*}_{(T,t)}(y) \right)} \left( \Phi^{v,u^*}_{(T,t)}(x) \right)^{\top} p \, }{\R^{2d}}{\nu^*_{T,N}(t)(y,p)} \\
& = \INTDom{ \INTDom{\BGamma^v_{\left( t , \Phi^{v,u^*}_{(T,t)}(y) \right)} \left( \Phi^{v,u^*}_{(T,t)}(x) \right)^{\top} p \, }{\R^d}{\sigma_{y,N}^*(t)(p)}}{\R^d}{\mu^*(T)(y)} \\
& = \INTDom{\BGamma^v_{ \left( t , \Phi^{v,u^*}_{(T,t)}(y) \right) } \left( \Phi^{v,u^*}_{(T,t)}(x) \right)^{\top} \Psi^{y,N}_{(T,t)}(p) }{\R^{2d}}{\big( (\Id \times (-\nabla_{\mu} \Scal_N(\mu^*(T))))_{\#} \mu^*(T)  \big)(y,p)}.
\end{aligned}
\end{equation*}
This can in turn be reformulated into the more concise expression 
\begin{equation*}
\begin{aligned}
\derv{}{t}{} \INTDom{\xi(x,r)}{\R^{2d}}{\nu_N^*(t)(x,r)} & = \INTDom{\langle \nabla_x \xi(x,r) , v[\mu^*(t)](t,x) + u^*(t,x) \rangle}{\R^{2d}}{\nu_N^*(t)(x,r)} \\
& + \INTDom{\langle \nabla_r \xi(x,r) , \lambda^N_0 \nabla_{\mu} L(t,\mu^*(t),u^*(t))(x) \rangle}{\R^{2d}}{\nu_N^*(t)(x,r)} \\
& + \INTDom{\langle \nabla_r \xi(x,r) , \nabla_{\mu} \Ccal(t,\mu^*(t),\zeta^*_N(t),u^*(t))(x) \rangle}{\R^{2d}}{\nu_N^*(t)(x,r)} \\
& - \INTDom{\langle \nabla_r \xi(x,r) , \D_x v[\mu^*(t)](t,x)^{\top}r + \D_x u^*(t,x)^{\top}r + \BGamma^v [\nu_N^*(t)] \left( t , x \right) \rangle}{\R^{2d}}{\nu_N^*(t)(x,r)}, 
\end{aligned}
\end{equation*}
which by \eqref{eq:NonlocalPDE_distributions2} precisely corresponds to the fact that $\nu^*_N(\cdot)$ is a solution of \eqref{eq:HamiltonianFlow_General}.
\end{proof}

\begin{rmk}[Wasserstein and classical costates]
In the finite-dimensional proof of the Gamkrelidze PMP, the optimal costates $p_N^*(\cdot)$ are defined as the solutions of the backward equations
\begin{equation}
\label{eq:ClassicalCostateIllustration}
\dot p_N^*(t) = -\nabla_x \Hpazo_{\lambda_0^N}(t,x^*_N(t),p_N^*(t),\zeta^*_N(t),u^*(t)), \qquad p_N^*(T) = -\nabla \pazocal{S}_N(x^*(T)), 
\end{equation}
where $\Hpazo_{\lambda_0^N}(\cdot,\cdot,\cdot,\cdot,\cdot)$ and $\Spazo_N(\cdot)$ are the counterparts of $\Hcal_{\lambda_0^N}(\cdot,\cdot,\cdot,\cdot)$ and $\Scal_N(\cdot)$ associated with the finite-dimensional optimal control problem
\begin{equation*}
\left\{
\begin{aligned}
\underset{u \in \U}{\min} & \left[ \INTSeg{L(t,x(t),u(t))}{t}{0}{T} + \varphi(x(T)) \right], \\
\textnormal{s.t.} 
& \left\{ 
\begin{aligned}
& \dot x(t) = v(t,x(t)) + u(t), \\
& x(0) = x^0 \in \R^d,
\end{aligned}
\right. \\
\textnormal{and}
& \left\{
\begin{aligned}
\Psi^I(x(T)) & \leq 0,~ \Psi^E(x(T)) = 0 ,  \\
\Lambda(t,x(t)) & \leq 0 ~~ \text{for all $t \in [0,T]$, }
\end{aligned}
\right.
\end{aligned}
\right.
\end{equation*}
and defined by 
\begin{equation*}
\left\{
\begin{aligned}
& \Hpazo_{\lambda_0^N}(t,x,p,\zeta,\omega) = \langle p , v(t,x) + \omega \rangle - \lambda^N_0 L(t,x,\omega) - \Cpazo(t,x,\zeta,\omega), \\
& \Spazo_N(x) = \lambda_0^N \varphi(x) + \sum_{i=1}^n \lambda_i^N \Psi_i^I(x) + \sum_{j=1}^m \eta_j^N \Psi_j^E(x).
\end{aligned}
\right.
\end{equation*}

In our statement of the PMP, one should think of $\pi^2_{\#} \nu^*(\cdot)$ as being concentrated on the characteristic curves of the backward costate dynamics. Indeed in Lemma \ref{lem:Wellposedness_General}, the curves $\sigma^*_{x,N}(\cdot)$ are concentrated on the unique characteristic of the linearized backward non-local dynamics \eqref{eq:Marginal_TransportEq_General} starting from $(-\nabla_{\mu} \Scal_N(\mu^*(T))(x))$ for $\mu^*(T)$-almost every $x \in \R^d$. The second marginal of the curve $\nu^*_{T,N}(\cdot) = \int \sigma^*_{x,N}(\cdot)\mu^*(T)(x)$ can then be seen as a Lagrangian superposition of integral curves of \eqref{eq:ClassicalCostateIllustration} depending on the starting point of the curve in $\supp(\mu^*(T))$. 
\end{rmk}

Now that we have built a suitable notion of solution for \eqref{eq:HamiltonianFlow_General}, let us prove that $\nu^*_N(\cdot)$ is such that the PMP holds with a relaxed maximization condition formulated over the collection of needle parameters $\{(\omega_k,\tau_k) \}_{k=1}^N \subset U_D \times \Acal$. 

\begin{lem}[A partial Pontryagin Maximum Principle]
\label{lem:Partial_Maximization_General}
For any $N \geq 1$, the curve of measures $\nu^*_N(\cdot)$ introduced in Lemma \ref{lem:Wellposedness_General} is a solution of the Hamiltonian flow \eqref{eq:PMP_HamiltonianFlow_General} associated to the Hamiltonian $\Hcal_{\lambda_0^N}(\cdot,\cdot,\cdot,\cdot)$ and to the Lagrange multipliers $(\lambda_0^N,\dots,\lambda_n^N,\dots,\eta_1^N,\dots,\eta_m^N,\varpi_1^N,\dots,\varpi_r^N) \in \{0,1\} \times \R_+^n \times \R^m \times \M_+([0,T])^r$. Moreover, the relaxed maximization condition 
\begin{equation}
\label{eq:Partial_Maximization_General}
\begin{aligned}
& \Hcal_{\lambda_0^N}(\tau_k,\nu_N^*(\tau_k),\zeta_N^*(\tau_k),\omega_k) - \sum_{l = 1}^r \varpi_l^N(\{ \tau_k \}) \INTDom{\langle \nabla_{\mu} \Lambda_l(\tau_k,\mu^*(\tau_k))(x) , \omega_k \rangle}{\R^d}{\mu^*(\tau_k)(x)} \\
\leq ~ & \Hcal_{\lambda_0^N}(\tau_k,\nu_N^*(\tau_k),\zeta_N^*(\tau_k),u^*(\tau_k)) - \sum_{l = 1}^r \varpi_l^N(\{ \tau_k \}) \INTDom{\langle \nabla_{\mu} \Lambda_l(\tau_k,\mu^*(\tau_k))(x) , u^*(\tau_k,x) \rangle}{\R^d}{\mu^*(\tau_k)(x)}
\end{aligned}
\end{equation}
holds for all $\{(\omega_k,\tau_k)\}_{k=1}^N \subset U_D \times \Acal$. 
\end{lem}

\begin{proof}
Using the expression \eqref{eq:Hamiltonian_General} of the augmented infinite-dimensional Hamiltonian along with the definition \eqref{eq:WassersteinDiff} of Wasserstein gradients and the results given in Proposition \ref{prop:Examp_Cost} and Proposition \ref{prop:WassersteinDiff_StateConstraints}, one can check that  
\begin{equation*}
\V^*_N[\nu^*_N(t)](t,x,r) = \J_{2d} \nabla_{\nu} \Hcal_{\lambda_0^N}(t,\nu^*_N(t),\zeta^*_N(t),u^*(t))(x,r), 
\end{equation*}
for $\Lcal^1$-almost every $t \in [0,T]$ and every $(x,r) \in \overline{B_{2d}(0,R'_T)}$.

For $k \in \{1,\dots,N\}$, we introduce the collection of maps $\K^N_{\omega_k,\tau_k}(\cdot)$ defined for $\Lcal^1$-almost all $t \in [\tau_k,T]$ by 
\begin{equation}
\label{eq:Kcal_Def}
\begin{aligned}
\K^N_{\omega_k,\tau_k}(t) & = \INTDom{\big\langle r , \F^{\omega_k,\tau_k}_t \circ \Phi^{v,u^*}_{(t,\tau_k)}(x) \big\rangle}{\R^{2d}}{\nu_N^*(t)(x,r)} + \lambda^N_0 \Big( L(\tau_k,\mu^*(\tau_k),\omega_k) - L(\tau_k,\mu^*(\tau_k),u^*(\tau_k)) \Big) \\
& - \INTSeg{ \INTDom{\left\langle  \lambda_0^N \nabla_{\mu} L(s,\mu^*(s),u^*(s))(x) , \F^{\omega_k,\tau_k}_s \circ \Phi^{v,u^*}_{(s,\tau_k)}(x) \right\rangle}{\R^d}{\mu^*(s)(x)}}{s}{\tau_k}{t} \\
& -\sum_{l=1}^r \INTSeg{\INTDom{\left\langle \nabla_{\mu} \Lambda_l(s,\mu^*(s))(x) , \F^{\omega_k,\tau_k}_s \circ \Phi^{v,u^*}_{(s,\tau_k)}(x) \right\rangle}{\R^d}{\mu^*(s)(x)}}{\varpi_l^N(s)}{\tau_k}{t} \\
& - \sum_{l=1}^r \zeta_{l,N}^*(t) \INTDom{\left\langle \nabla_{\mu} \Lambda_l(t,\mu^*(t))(x) , \F^{\omega_k,\tau_k}_t \circ \Phi^{v,u^*}_{(t,\tau_k)}(x) \right\rangle}{\R^d}{\mu^*(t)(x)}.
\end{aligned}
\end{equation}
$k \in \{1,\dots,N\}$. By construction, the maps $\K^N_{\omega_k,\tau_k}(\cdot)$ satisfy
\begin{equation}
\label{eq:Evaluation_Kfunctions1}
\K^N_{\omega_k,\tau_k}(T) \leq 0, 
\end{equation}
since it can be checked that the evaluation of $\K^N_{\omega_k,\tau_k}(\cdot)$ at $T$ coincides with the left-hand side of \eqref{eq:OptimalityConditions_General}, which has been shown to be non-positive for all $k \in \{1,\dots,N\}$. Moreover, the evaluation of the maps $\K^N_{\omega_k,\tau_k}(\cdot)$ at $\tau_k$ can be written explicitly as 
\begin{equation}
\label{eq:Evaluation_Kfunctions2}
\begin{aligned}
\K^N_{\omega_k,\tau_k}(\tau_k) & = \Hcal_{\lambda_0^N}(\tau_k,\nu^*_N(\tau_k),\zeta^*_N(\tau_k),\omega_k) - \Hcal_{\lambda_0^N}(\tau_k,\nu^*_N(\tau_k),\zeta^*_N(\tau_k),u^*(\tau_k)) \\
& - \sum_{l =1}^r \varpi_l^N(\{\tau_k\}) \INTDom{\langle \nabla_{\mu} \Lambda_l(\tau_k,\mu^*(\tau_k)) , \omega_k(x) - u^*(\tau_k,x) \rangle}{\R^d}{\mu^*(\tau_k)(x)}.
\end{aligned}
\end{equation}

We now aim at showing that the maps $\K^N_{\omega_k,\tau_k}(\cdot)$ are constant over $[\tau_k,T]$. By definition, these functions are in $BV([0,T],\R)$ and therefore admit a distributional derivative in the form of finite Borel regular measures (see e.g. \cite[Chapter 3]{AmbrosioFuscoPallara}). A simple computation of the time derivatives of the last two terms in \eqref{eq:Kcal_Def} shows that the non-absolutely continuous parts of the derivatives of the maps $\K^N_{\omega_k,\tau_k}(\cdot)$ cancel each other out, since the weak derivatives of the maps $\zeta^*_{N,l}(\cdot)$ are such that $\text{d}\zeta^*_{N,l} = -\varpi_l^N$. Hence, the maps $\K^N_{\omega_k,\tau_k}(\cdot)$ are absolutely continuous and therefore differentiable $\Lcal^1$-almost everywhere. One can then compute their derivative at $\Lcal^1$-almost every $t \in [\tau_k,T]$ as follows. 
\begin{equation}
\label{eq:Derivative_Kfunction_General}
\begin{aligned}
\derv{}{t}{} \K^N_{\omega_k,\tau_k}(t) & =  \INTDom{\INTDom{\big\langle r , \partial_t \F^{\omega_k,\tau_k}_t \circ \Phi^{v,u^*}_{(T,\tau_k)}(x) \big\rangle}{\R^d}{\sigma^*_{x,N}(t)(r)}}{\R^d}{\mu^*(T)(x)}  \\
& + \INTDom{\INTDom{\big\langle \V_N^*[\nu_N^*(t)](t,x,r) , \F^{\omega_k,\tau_k}_t \circ \Phi^{v,u^*}_{(T,\tau_k)}(x) \big\rangle}{\R^d}{\sigma^*_{x,N}(t)(r)}}{\R^d}{\mu^*(T)(x)} \\
& - \INTDom{\left\langle \lambda_0^N \nabla_{\mu} L(t,\mu^*(t),u^*(t)) \left( \Phi^{v,u^*}_{(T,t)}(x) \right) , \F^{\omega_k,\tau_k}_t \circ \Phi^{v,u^*}_{(T,\tau_k)}(x) \right\rangle}{\R^d}{\mu^*(T)(x)} \\
& - \sum_{l=1}^r \zeta^*_{l,N}(t) \derv{}{t}{} \left[ \INTDom{\left\langle \nabla_{\mu} \Lambda_l(t,\mu^*(t)) \left( \Phi^{v,u^*}_{(T,t)}(x) \right) , \F^{\omega_k,\tau_k}_t \circ \Phi^{v,u^*}_{(T,\tau_k)}(x) \right\rangle}{\R^d}{\mu^*(T)(x)} \right]
\end{aligned}
\end{equation}
The time-derivatives of the summands of the last term can be computed as follows using Proposition \ref{prop:Chainrule} and the geometric structure \eqref{eq:Flow_def} of solutions of \eqref{eq:Nonlocal_Cauchy_problem} associated with the non-local velocity field $v[\mu^*(t)](t,\cdot) + u^*(t,\cdot)$.
\begin{equation*}
\begin{aligned}
\derv{}{t}{} \Bigg[ & \INTDom{\left\langle \nabla_{\mu} \Lambda_l(t,\mu^*(t)) \left( \Phi^{v,u^*}_{(T,t)}(x) \right) , \F^{\omega_k,\tau_k}_t \circ \Phi^{v,u^*}_{(T,\tau_k)}(x) \right\rangle}{\R^d}{\mu^*(T)(x)} \Bigg] \\
= ~ & \INTDom{\left\langle \nabla_{\mu} \Lambda_l(t,\mu^*(t)) \left( \Phi^{v,u^*}_{(T,t)}(x) \right) , \partial_t \F^{\omega_k,\tau_k}_t \circ \Phi^{v,u^*}_{(T,\tau_k)}(x) \right\rangle}{\R^d}{\mu^*(T)(x)}  \\
+ ~ & \INTDom{\left\langle \partial_t \nabla_{\mu} \Lambda_l(t,\mu^*(t)) \left( \Phi^{v,u^*}_{(T,t)}(x) \right) , \F^{\omega_k,\tau_k}_t \circ \Phi^{v,u^*}_{(T,\tau_k)}(x) \right\rangle}{\R^d}{\mu^*(T)(x)} \\
+ ~ & \INTDom{\left\langle \D_x \nabla_{\mu} \Lambda_l(t,\mu^*(t)) \left( \Phi^{v,u^*}_{(T,t)}(x) \right) v[\mu^*(t)] \left( t , \Phi^{v,u^*}_{(T,t)}(x) \right) , \F^{\omega_k,\tau_k}_t \circ \Phi^{v,u^*}_{(T,\tau_k)}(x) \right\rangle}{\R^d}{\mu^*(T)(x)} \\
+ ~ & \INTDom{\left\langle \D_x \nabla_{\mu} \Lambda_l(t,\mu^*(t)) \left( \Phi^{v,u^*}_{(T,t)}(x) \right) u^* \left( t , \Phi^{v,u^*}_{(T,t)}(x) \right) , \F^{\omega_k,\tau_k}_t \circ \Phi^{v,u^*}_{(T,\tau_k)}(x) \right\rangle}{\R^d}{\mu^*(T)(x)} \\
+ ~ & \INTDom{\left\langle \INTDom{\BGamma^{\nabla_{\mu} \Lambda_l}_{\left( t , \Phi^{v,u^*}_{(T,t)}(x) \right)} \left( \Phi^{v,u^*}_{(T,t)}(y) \right) v[\mu^*(t)] \left( t , \Phi^{v,u^*}_{(T,t)}(y) \right)}{\R^d}{\mu^*(T)(y)} , \F^{\omega_k,\tau_k}_t \circ \Phi^{v,u^*}_{(T,\tau_k)}(x) \right\rangle}{\R^d}{\mu^*(T)(x)} \\
+ ~ & \INTDom{\left\langle \INTDom{\BGamma^{\nabla_{\mu} \Lambda_l}_{\left( t , \Phi^{v,u^*}_{(T,t)}(x) \right)} \left( \Phi^{v,u^*}_{(T,t)}(y) \right)  u^* \left( t , \Phi^{v,u^*}_{(T,t)}(y) \right)}{\R^d}{\mu^*(T)(y)} , \F^{\omega_k,\tau_k}_t \circ \Phi^{v,u^*}_{(T,\tau_k)}(x) \right\rangle}{\R^d}{\mu^*(T)(x)} \\
= ~ & \INTDom{\left\langle \nabla_{\mu} \Ccal_l(t,\mu^*(t),u^*(t)) \left( \Phi^{v,u^*}_{(T,t)}(x) \right) , \F^{\omega_k,\tau_k}_t \circ \Phi^{v,u^*}_{(T,\tau_k)}(x) \right\rangle}{\R^d}{\mu^*(T)(x)}
\end{aligned}
\end{equation*}
by applying Fubini's Theorem and identifying the analytical expressions of the Wasserstein gradients of the summands $\Ccal_l(t,\mu^*(t),u^*(t))$ of $\Ccal(t,\mu^*(t),\zeta^*(t),u^*(t))$ derived in Proposition \ref{prop:WassersteinDiff_StateConstraints}. Plugging this expression into \eqref{eq:Derivative_Kfunction_General} along with the characterization \eqref{eq:FirstOrder_Needle_Expression_General} of $\partial_t \F^{\omega_k,\tau_k}_t(\cdot)$ derived in Lemma \ref{lem:First_order_needle_General}, we obtain that
\begin{equation*}
\begin{aligned}
\derv{}{t}{} \K^N_{\omega_k,\tau_k}(t) & =  \INTDom{\Big\langle \D_x u^*(t,x)^{\top}r + \D_x v[\mu^*(t)](t,x)^{\top}r + \BGamma^v[\nu^*_N(t)](t,x) , \F^{\omega_k,\tau_k}_t \circ \Phi^{v,u^*}_{(t,\tau_k)}(x) \Big\rangle}{\R^{2d}}{\nu_N^*(t)(x,r)}  \\
& + \INTDom{\big\langle \V_N^*[\nu_N^*(t)](t,x,r) , \F^{\omega_k,\tau_k}_t \circ \Phi^{v,u^*}_{(t,\tau_k)}(x) \big\rangle}{\R^{2d}}{\nu_N^*(t)(x)} \\
& - \INTDom{\left\langle \lambda_0^N \nabla_{\mu} L(t,\mu^*(t),u^*(t))(x) , \F^{\omega_k,\tau_k}_t \circ \Phi^{v,u^*}_{(t,\tau_k)}(x) \right\rangle}{\R^d}{\mu^*(t)(x)} \\
& - \sum_{l=1}^r \zeta^*_{l,N}(t)\INTDom{\left\langle \nabla_{\mu} \Ccal_l(t,\mu^*(t),u^*(t))(x) , \F^{\omega_k,\tau_k}_t \circ \Phi^{v,u^*}_{(T,t)}(x) \right\rangle}{\R^d}{\mu^*(t)(x)}. \\
\end{aligned}
\end{equation*}
In the first line we used the fact that
\begin{equation*}
\begin{aligned}
& \INTDom{\left\langle r , \INTDom{\BGamma^v_{(t,x)}(y) \F^{\omega_k,\tau_k}_t \circ \Phi^{v,u^*}_{(t,\tau_k)}(y)}{\R^d}{\mu^*(t)(y)} \right\rangle}{\R^{2d}}{\nu^*_N(t)(x,r)} \\
= ~ & \INTDom{\INTDom{\left\langle \BGamma^v_{(t,y)}(x)^{\top} \, p , \F^{\omega_k,\tau_k}_t \circ \Phi^{v,u^*}_{(t,\tau_k)}(x) \right\rangle}{\R^{2d}}{\nu^*_N(t)(x,r)}}{\R^d}{\mu^*(t)(y)} \\
= ~ & \INTDom{ \left\langle \BGamma^v[\nu^*_N(t)](t,x) , \F^{\omega_k,\tau_k}_t \circ \Phi^{v,u^*}_{(t,\tau_k)}(x) \right\rangle }{\R^d}{\mu^*(t)(x)}. 
\end{aligned}
\end{equation*}
as a consequence of  Fubini's Theorem, where $\BGamma^v[\nu^*_N(\cdot)](\cdot,\cdot)$ is defined as in \eqref{eq:BGammaBis_def}. Recalling the definition of the vector field $\V_N^*[\cdot](\cdot,\cdot,\cdot)$ given in \eqref{eq:HamiltonianFlow_General}, we therefore observe that $\derv{}{t}{} \K^N_{\omega_k,\tau_k}(t) = 0$ for $\Lcal^1$-almost every $t \in [\tau_k,T]$, so that it is constant over this time interval. Merging this fact with \eqref{eq:Evaluation_Kfunctions1} and \eqref{eq:Evaluation_Kfunctions2} yields \eqref{eq:Partial_Maximization_General} and concludes the proof of our claim.
\end{proof}


\bigskip \newpage

\begin{flushleft}
\underline{\textbf{Step 4 : Limiting procedure}}
\end{flushleft}

\smallskip

\subsubsection*{\textbf{The PMP for absolutely continuous state constraints multipliers}}

In Step 3, we have built for any $N \geq 1$ a suitable state-costate curve $\nu^*_N(\cdot)$ solution of the Hamiltonian system \eqref{eq:PMP_HamiltonianFlow_General}, and such that the relaxed Pontryagin maximization condition \eqref{eq:Partial_Maximization_General} holds on an $N$-dimensional subset of needle parameters. The last step in the proof of Theorem \ref{thm:PMP_General} is to take the limit as $N$ goes to infinity of the previous optimality conditions in order to recover the PMP formulated on the whole set of needle parameters. 

By the non-triviality condition (NT), the sequence of Lagrange multipliers $\big( (\lambda_0^N,\dots,\lambda_n^N,\eta_1^N,\dots,\eta_m^N) \big) \subset \{0,1\} \times \R_+^n \times \R^m$ is bounded uniformly with respect to $N$. Hence, we can extract a subsequence of multipliers such that $(\lambda_0^N,\dots,\lambda_n^N,\eta_1^N,\dots,\eta_m^N) \longrightarrow (\lambda_0,\dots,\lambda_n,\eta_1,\dots,\eta_m)$ as $N \rightarrow + \infty$. A straightforward passage to the limit shows that these limit multipliers satisfy the complementary slackness condition
\begin{equation*}
\lambda_i \Psi^I_i(\mu^*(T)) = 0 \qquad \text{for all $i \in \{1,\dots,n\}$}.
\end{equation*}
Similarly, the sequence of measures $(\varpi_1^N,\dots,\varpi_r^N) \subset \M_+([0,T])^r$ is uniformly bounded with respect to the total variation norm as a consequence of Lemma \ref{lem:Derivative_RCSC}. By Banach-Alaoglu's Theorem (see e.g. \cite[Theorem 3.16]{Brezis}), it therefore admits a weakly-$^*$ converging subsequence to some $(\varpi_1,\dots,\varpi_r) \in \M_+([0,T])^r$. Moreover for any $l \in \{1,\dots,r\}$, the measures $(\varpi_l^N)$ are equi-supported in the sets $\{ t \in [0,T] ~\text{s.t.}~ \Lambda_l(t,\mu^*(t)) =0 \}$. It can then be shown by standard convergence properties on the supports of sequences of measures (see e.g. \cite[Proposition 5.1.8]{AGS}) that the limit measures $(\varpi_1,\dots,\varpi_r)$ are supported on these sets as well. Therefore, it holds that
\begin{equation*}
\supp(\varpi_l) \subseteq \Big\{ t \in [0,T] ~\text{s.t.}~ \Lambda_l(t,\mu^*(t)) = 0 \Big\},
\end{equation*}
for any $l \in \{1,\dots,r\}$. Furthermore, if all the scalar Lagrange multipliers $(\lambda_0^N,\dots,\lambda_n^N,\eta_1^N,\dots,\eta_m^N)$ vanish for large $N$, it follows from the non-triviality condition (NT) that $\Norm{\varpi_l^N}_{TV} > 0$ so that $\varpi_l \neq 0$ in the limit, at least for some $l \in \{1,\dots,r\}$. Hence, we recover the non-degeneracy condition \eqref{eq:NonDegeneracy_General}, i.e.
\begin{equation*}
(\lambda_0,\dots,\lambda_n,\eta_1,\dots,\eta_m,\varpi_1,\dots,\varpi_r) \neq 0.
\end{equation*}
In Lemma \ref{lem:Wellposedness_General}, we have shown that the curves of measures $\nu^*_N(\cdot)$ are uniformly equi-compactly supported and equi-Lipschitz. Hence, $(\nu^*_N(\cdot))$ admits converging subsequences in the $C^0$-topology by Ascoli-Arzel\`a Theorem (see e.g. \cite[Theorem 11.28]{Rudin1987}). 

We now prove that there exists an accumulation point $\nu^*(\cdot)$ of $(\nu^*_N(\cdot))$ which solves the system of equations \eqref{eq:PMP_HamiltonianFlow_General} associated with the limit multipliers $(\lambda_0,\dots,\lambda_n,\eta_1,\dots,\eta_m,\varpi_1,\dots,\varpi_r)$. To this end, we start by making an extra simplifying assumption on the state constraints multipliers. 

\begin{framed}
\begin{center}
\vspace{-0.15cm}
\textnormal{\textbf{(H7)} :} The measures $(\varpi_1,\dots,\varpi_r)$ are absolutely continuous with respect to $\Lcal^1$ on $[0,T]$.
\end{center}
\vspace{-0.15cm}
\end{framed}
We shall see in the sequel how this extra assumption can be lifted at the price of an extra approximation argument by absolutely continuous measures. Let $\nu^*(\cdot) \in \Lip([0,T],\Pcal(\overline{B_{2d}(0,R'_T)})$ be an accumulation point of $(\nu^*_N(\cdot))$ along a suitable subsequence. As a direct consequence of the convergence of the scalar Lagrange multipliers, one recovers the uniform convergence of the final gradient map
\begin{equation*}
\nabla_{\mu} \Scal_N(\mu^*(T))(\cdot) ~\underset{N \rightarrow +\infty}{\overset{C^0}{\longrightarrow}}~ \nabla_{\mu} \Scal(\mu^*(T))(\cdot).
\end{equation*}
This implies by standard convergence results for pushforwards of measures (see e.g. \cite[Lemma 5.2.1]{AGS}) that $\nu^*(\cdot)$ satisfies the boundary condition
\begin{equation*}
\pi^2_{\#}\nu^*(T) = (-\nabla_{\mu} \Scal(\mu^*(T)))_{\#} \mu^*(T). 
\end{equation*}

Moreover, the weak-$^*$ convergence of $(\varpi_1^N,\dots,\varpi_r^N)$ towards $(\varpi_1,\dots,\varpi_r)$ along with \textnormal{\textbf{(H7)}} implies by Proposition \ref{prop:Portmanteau} that 
\begin{equation*}
\zeta^*_{l,N}(t) = \mathds{1}_{[0,T)}(t) \INTSeg{}{\varpi_l^N(s)}{t}{T} ~\underset{N \rightarrow +\infty}{\longrightarrow}~ \mathds{1}_{[0,T)}(t) \INTSeg{}{\varpi_l(s)}{t}{T} = \zeta^*_l(t)
\end{equation*}
for all times $t \in [0,T]$. By definition \eqref{eq:NonlocalPDE_distributions1} of distributional solutions to transport equations, the fact that $\nu^*_N(\cdot)$ is a solution of \eqref{eq:HamiltonianFlow_General} can be written as
\begin{equation}
\label{eq:Distribution_General}
\INTSeg{\INTDom{\Big( \partial_t \xi(t,x,r) + \left\langle \nabla_{(x,r)} \xi(t,x,r) , \J_{2d} \nabla_{\nu}\Hcal_{\lambda_0^N}(t,\nu^*_N(t),\zeta^*_N(t),u^*(t))(x,r) \right\rangle \Big)}{\R^{2d}}{\nu^*_N(t)(x,r)}}{t}{0}{T} = 0
\end{equation}
for any $\xi \in C^{\infty}_c([0,T] \times \R^{2d})$. Since all the functionals involved in the definition of the Wasserstein gradient of the augmented infinite-dimensional Hamiltonian are continuous and bounded, we have that 
\begin{equation*}
\nabla_{\nu}\Hcal_{\lambda_0^N}(t,\nu^*_N(t),\zeta^*_N(t),u^*(t))(\cdot,\cdot) ~\underset{N \rightarrow +\infty}{\overset{C^0}{\longrightarrow}}~ \nabla_{\nu}\Hcal_{\lambda_0}(t,\nu^*(t),\zeta^*(t),u^*(t))(\cdot,\cdot)
\end{equation*}
uniformly with respect to $t \in [0,T]$, as a by-product of the convergence of the Lagrange multipliers. By using this fact along with the uniform equi-compactness of the supports of $(\nu^*_N(\cdot))$, we can take the limit as $N \rightarrow +\infty$ in \eqref{eq:Distribution_General} an apply Lebesgue's Dominated Convergence Theorem to recover that
\begin{equation*}
\INTSeg{\INTDom{\Big( \partial_t \xi(t,x,r) + \left\langle \nabla_{(x,r)} \xi(t,x,r) , \nabla_{\nu}\Hcal_{\lambda_0}(t,\nu^*(t),\zeta^*(t),u^*(t))(x,r) \right\rangle \Big)}{\R^{2d}}{\nu^*(t)(x,r)}}{t}{0}{T} = 0
\end{equation*}
for any $\xi \in C^{\infty}_c([0,T] \times \R^{2d})$. Hence, the accumulation point $\nu^*(\cdot)$ of $(\nu^*_N(\cdot))$ in the $C^0$-topology is a solution of the Hamiltonian flow \eqref{eq:HamiltonianFlow_General} associated with the limit multipliers $(\lambda_0,\dots,\lambda_n,\eta_1,\dots,\eta_m,\varpi_1,\dots,\varpi_r)$. 

In order to complete our proof of Theorem \ref{thm:PMP_General}, there remains to show that the limit curve $\nu^*(\cdot)$ is such that the maximization condition \eqref{eq:Maximization_General} holds for $\Lcal^1$-almost every $t \in [0,T]$. We know that for any $(\omega_k,\tau_k) \in U_D \times \Acal$, the modified maximization condition \eqref{eq:Partial_Maximization_General} holds. By the preliminary assumption \textbf{(H7)} that the limit measures $(\varpi_1,\dots,\varpi_r)$ are absolutely continuous with respect to $\Lcal^1$, we can apply Proposition \ref{prop:Portmanteau} to recover that 
\begin{equation*}
\varpi_l^N(\{ \tau_k \}) ~\underset{N \rightarrow +\infty}{\longrightarrow}~ \varpi_l(\{ \tau_k \}) = 0,
\end{equation*}
for any $l \in \{ 1,\dots,r \}$. Since the infinite-dimensional Hamiltonian is continuous with respect to its second argument in the $W_1$-metric and linear with respect to its third argument, it holds that 
\begin{equation*}
\Hcal_{\lambda_0^N}(\tau_k,\nu^*_N(\tau_k),\zeta^*_N(\tau_k),\omega_k) ~\underset{N \rightarrow +\infty}{\longrightarrow}~ \Hcal_{\lambda_0}(\tau_k,\nu^*(\tau_k),\zeta^*(\tau_k),\omega_k)
\end{equation*}
and 
\begin{equation*}
\Hcal_{\lambda_0^N}(\tau_k,\nu^*_N(\tau_k),\zeta^*_N(\tau_k),u^*(\tau_k)) ~\underset{N \rightarrow +\infty}{\longrightarrow}~ \Hcal_{\lambda_0}(\tau_k,\nu^*(\tau_k),\zeta^*(\tau_k),u^*(\tau_k))
\end{equation*}
uniformly with respect to $k \leq N$.  We can therefore pass to the limit as $N \rightarrow +\infty$ in the partial maximization condition \eqref{eq:Partial_Maximization_General} to obtain that
\begin{equation}
\label{eq:Total_Maximization_Proof1}
\Hcal_{\lambda_0}(\tau_k,\nu^*(\tau_k),\zeta^*(\tau_k),\omega_k) \leq \Hcal_{\lambda_0}(\tau_k,\nu^*(\tau_k),\zeta^*(\tau_k),u^*(\tau_k)) 
\end{equation}
for any $(\omega_k,\tau_k) \in U_D \times \Acal$. 

Given an arbitrary pair $(\omega,\tau) \in U \times \Mcal$, it is possible to choose a sequence of elements $\{(\omega_k,\tau_k)\}_k \subset U_D \times \Acal$ such that 
\begin{equation*}
\tau_k \underset{k \rightarrow +\infty}{\longrightarrow} \tau, \qquad \omega_k \overset{C^0}{\underset{k \rightarrow +\infty}{\longrightarrow}} \omega,
\end{equation*}
and 
\begin{equation}
\label{eq:Recall_Lusin_General}
\NormC{u^*(\tau,\cdot) - u^*(\tau_k,\cdot)}{0}{\overline{B(0,R_T)},\R^d} ~\underset{k \rightarrow +\infty}{\longrightarrow}~ 0, \qquad \NormC{L(\tau,\mu^*(\tau),\cdot)-L(\tau_k,\mu^*(\tau_k),\cdot)}{0}{U,\R} ~\underset{k \rightarrow +\infty}{\longrightarrow}~ 0.
\end{equation}
Remark first that under assumption \textnormal{\textbf{(H7)}}, the maps $t \rightarrow \zeta^*_l(t)$ are continuous on $[0,T)$. By \eqref{eq:Recall_Lusin_General} along with the continuity of the augmented infinite-dimensional Hamiltonian in the $C^0$-norm topology with respect to its fourth argument, we can pass to the limit as $k \rightarrow +\infty$ in \eqref{eq:Total_Maximization_Proof1}. This yields the Pontryagin Maximization condition 
\begin{equation*}
\Hcal_{\lambda_0}(\tau,\nu^*(\tau),\zeta^*(\tau),\omega) \leq \Hcal_{\lambda_0}(\tau,\nu^*(\tau),\zeta^*(\tau),u^*(\tau)) 
\end{equation*}
for any pair $(\omega,\tau) \in U \times \Mcal$. 

\smallskip

\subsubsection*{\textbf{Lifting the absolute continuity hypothesis \textnormal{\textbf{(H7)}}}}

\smallskip

In order to recover the full statement of Theorem \ref{thm:PMP_General}, we now show how to relax the absolute continuity assumption \textbf{(H7)} made on the state constraints multipliers. For a given small parameter $\epsilon > 0$, we consider a sequence of mollifiers $\rho_{\epsilon} : t \in [0,T] \mapsto \epsilon^{-1} \rho(t/\epsilon)$ where $\rho \in C^{\infty}_c([0,T])$ is such that $\INTSeg{\rho(t)}{t}{0}{T} = 1$. Given $N \geq 1$, we define the mollified measure $(\varpi_1^{N,\epsilon},\dots,\varpi_r^{N,\epsilon})$ by
\begin{equation*}
\varpi_l^{N,\epsilon} = (\rho_{\epsilon} * \varpi^N_l)(t) \cdot \Lcal^1  
\end{equation*}
where for any $l \in \{1,\dots,r\}$, the convolution maps are defined by $\rho_{\epsilon} * \varpi^N_l : t \in [0,T] \mapsto \INTSeg{\rho_{\epsilon}(t-s)}{\varpi^N_l(s)}{0}{T}$. Using the fact that the functions 
\begin{equation*}
t \in [\tau_k,T] \mapsto \INTDom{\langle \nabla_{\mu} \Lambda_l(t,\mu^*(t))(x) , \F^{\omega_k,\tau_k}_t \circ \Phi^{v,u^*}_{(t,\tau_k)}(x)  \rangle}{\R^d}{\mu^*(t)(x)}
\end{equation*}
are Lipschitz and bounded as a by-product of \textbf{(H6)} and Lemma \ref{lem:First_order_needle_General}, one can assert using the definition of the dual norm in the Banach space $\M_+([0,T])$ that 
\begin{equation*}
\begin{aligned}
& - \INTSeg{\INTDom{\langle \nabla_{\mu} \Lambda_l(t,\mu^*(t))(x) , \F^{\omega_k,\tau_k}_t \circ \Phi^{v,u^*}_{(t,\tau_k)}(x)  \rangle}{\R^d}{\mu^*(t)(x)}}{\varpi_l^N(t)}{\tau_k}{T} \\
\geq & - \INTSeg{\INTDom{\langle \nabla_{\mu} \Lambda_l(t,\mu^*(t))(x) , \F^{\omega_k,\tau_k}_t \circ \Phi^{v,u^*}_{(t,\tau_k)}(x)  \rangle}{\R^d}{\mu^*(t)(x)}}{\varpi_l^{N,\epsilon}(t)}{\tau_k}{T} - C\epsilon
\end{aligned}
\end{equation*}
for some uniform constant $C > 0$. This allows us to rewrite the optimality conditions \eqref{eq:OptimalityConditions_General} derived at time $T$ as 
\begin{equation}
\begin{aligned}
& \INTDom{\langle -\nabla_{\mu} \Scal_N(\mu^*(T))(x) , \F^{\omega_k,\tau_k}_T \circ \Phi^{v,u^*}_{(T,\tau_k)}(x) \rangle}{\R^d}{\mu^*(T)(x)} - \lambda^0_N \Big( L(\tau_k,\mu^*(\tau_k),u^*(\tau_k)) - L(\tau_k,\mu^*(\tau_k),\omega_k) \Big) \\
- & \INTSeg{\INTDom{ \left\langle \lambda^0_N \nabla_{\mu} L(t,\mu^*(t),u^*(t))(x) , \F^{\omega_k,\tau_k}_t \circ \Phi^{v,u^*}_{(t,\tau_k)}(x) \right\rangle}{\R^d}{\mu^*(t)(x)}}{t}{\tau_k}{T} \\
- & \sum_{l=1}^r \INTSeg{\INTDom{\left\langle \nabla_{\mu} \Lambda_l(t,\mu^*(t))(x) , \F^{\omega_k,\tau_k}_t \circ \Phi^{v,u^*}_{(t,\tau_k)}(x) \right\rangle}{\R^d}{\mu^*(t)(x)}}{\varpi_l^{N,\epsilon}(t)}{\tau_k}{T} \leq C \epsilon
\end{aligned}
\end{equation}
By defining the family of measure curves $(\nu^*_{N,\epsilon}(\cdot))$ as in Lemma \ref{lem:Wellposedness_General}, we can prove that the corresponding maps $\K^{N,\epsilon}_{\omega_k,\tau_k}(\cdot)$ defined as in \eqref{eq:Derivative_Kfunction_General} are constant over $[\tau_k,T]$ and that the partial maximization conditions 
\begin{equation*}
\Hcal_{\lambda_0^N}(\tau_k,\nu^*_{N,\epsilon}(\tau_k),\zeta^*_{N,\epsilon}(\tau_k),\omega_k) \leq \Hcal_{\lambda_0^N}(\tau_k,\nu^*_{N,\epsilon}(\tau_k),\zeta^*_{N,\epsilon}(\tau_k),u^*(\tau_k)) + C \epsilon
\end{equation*}
hold for any $\epsilon > 0$. Performing again the limiting arguments of Step 4 as $N \rightarrow +\infty$ and remarking that
\begin{equation*}
\varpi_l^{N,\epsilon} ~\underset{N \rightarrow +\infty}{\rightharpoonup^*}~ \varpi_l^{\epsilon} = (\rho_{\epsilon} * \varpi_l) \cdot \Lcal^1,
\end{equation*}
we recover the statement of the PMP with a measure curve $\nu^*_{\epsilon}(\cdot)$ associated to the Lagrange multipliers $(\lambda_0,\dots,\lambda_n,\eta_1,\dots,\eta_m,\varpi_1^{\epsilon},\dots,\varpi_r^{\epsilon})$. This limit curve is such that the relaxed maximization condition
\begin{equation}
\label{eq:Relaxed_Maximization}
\Hcal_{\lambda_0}(\tau,\nu^*_{\epsilon}(\tau),\zeta^*_{\epsilon}(\tau),\omega) \leq \Hcal_{\lambda_0}(\tau_k,\nu^*_{\epsilon}(\tau),\zeta^*_{\epsilon}(\tau),u^*(\tau)) + C \epsilon
\end{equation}
holds for any $(\omega,\tau) \in U \times \Mcal$. There now remains to perform a last limiting argument as $\epsilon \downarrow 0$ to recover the full maximum principle. 

By Lebesgue's Decomposition Theorem for finite Borel measures on the real line (see e.g. \cite[Remark 3.32, Corollary 3.33]{AmbrosioFuscoPallara}), the sets $\{ t \in [0,T] ~\text{s.t.}~ \varpi_l(\{t\}) > 0 \}$ are at most countable for all $l \in \{1,\dots,r\}$. Therefore, the set $\Mcal^{\circ} = \Mcal \backslash \bigcup_{l=1}^r \{ t \in [0,T] ~\text{s.t.}~ \varpi_l(\{t\}) > 0 \}$ has full Lebesgue measure in $[0,T]$, and by Proposition \ref{prop:Portmanteau} it is such that
\begin{equation*}
\zeta^*_{l,\epsilon}(\tau) = \INTSeg{}{\varpi_l^{\epsilon}(s)}{\tau}{T} ~\underset{\epsilon \downarrow 0}{\longrightarrow}~ \INTSeg{}{\varpi_l(s)}{\tau}{T} = \zeta^*_l(\tau). 
\end{equation*}
for any $\tau \in \Mcal^{\circ}$. By Ascoli-Arzel\`a Theorem, it holds that $\nu^*_{\epsilon}(\cdot) \rightarrow \nu^*(\cdot)$ in the $C^0$-norm topology and it can be checked that this limit curve solves the forward-backward system of continuity equations \eqref{eq:PMP_HamiltonianFlow_General} associated with the multipliers $(\lambda_0,\dots,\lambda_n,\eta_1,\dots,\eta_m,\varpi_1,\dots,\varpi_r)$. Moreover, letting $\epsilon \downarrow 0$ in \eqref{eq:Relaxed_Maximization} implies that the Pontryagin maximization condition \eqref{eq:Maximization_General} holds on the restricted subset $U \times \Mcal^{\circ}$.


\appendix


\section{Examples of functionals satisfying hypotheses \textnormal{\textbf{(H)}}}
\label{appendix:Examples}

In this Appendix, we show that the rather long list of hypotheses \textnormal{\textbf{(H)}} is not too restrictive and that a good score of relevant functionals for applications fit into the framework of Theorem \ref{thm:PMP_General}. This list of examples is partly borrowed from our previous work \cite{PMPWass}.

\begin{prop}[Example of non-local velocity field]
Let $(t,x,y) \mapsto H(t,x,y) \in \R^d$ be measurable with respect to $t \in [0,T]$, sublinear, Lipschitz and $C^1$-with respect to $(x,y) \in \R^{2d}$. Then, the map $\mu \in \Pcal_c(\R^d) \mapsto v[\mu](\cdot,\cdot)$ defined by 
\begin{equation*}
v[\mu](t,x) = \INTDom{H(t,x,y)}{\R^d}{\mu(y)}
\end{equation*}
for $\Lcal^1$-almost every $t \in [0,T]$ and any $x \in \R^d$ satisfies the hypotheses \textnormal{\textbf{(H2)}}-\textnormal{\textbf{(H3)}} of Theorem \ref{thm:PMP_General}. Moreover, its first-order variations $\D_x v[\mu](t,x)$ and $\INTDom{\BGamma^v_{(t,x)}(y)}{\R^d}{\mu(y)}$ are given by 
\begin{equation*}
\D_x v[\mu](t,x) = \INTDom{\D_x H(t,x,y)}{\R^d}{\mu(y)} ~~,~~
\INTDom{\BGamma^v_{(t,x)}(y)}{\R^d}{\mu(y)}  = \INTDom{\D_y H(t,x,y)}{\R^d}{\mu(y)}.
\end{equation*} 
\end{prop}

\begin{prop}[Example of cost and constraint functions]
\label{prop:Examp_Cost} 
Let $n \geq 1$ and $W \in C^1(\R^{nd},\R)$. Then, the functional 
\begin{equation*}
\varphi : \mu \in \Pcal_c(\R^d) \mapsto \INTDom{W(x_1,\dots,x_n)}{\R^{nd}}{\mu^{\otimes n}(x_1,\dots,x_n)}
\end{equation*}
with $\mu^{\otimes n} = \mu \times \dots \times \mu$ satisfies \textnormal{\textbf{(H4)}} of Theorem \ref{thm:PMP_General} and its Wasserstein gradient at some $\mu \in \Pcal(K)$ is given by 
\begin{equation*}
\nabla_{\mu} \varphi(\mu)(x_1, \dots ,x_n) = \sum_{j=1}^n \nabla_{x_j}W(x_1,\dots,x_n).
\end{equation*}

Let $m \in C^1(\R^d,\R^n)$ and $(t,x,v,r) \mapsto l(t,x,v,r) \in \R$ be $\Lcal^1$-measurable with respect to $t \in [0,T]$ and $C^1$-smooth with respect to $(x,v,r) \in \R^d \times \R^d \times \R^n$. Then, the functional 
\begin{equation*}
L : (t,\mu,\omega) \in [0,T] \times \Pcal(K) \times U \mapsto \INTDom{l \Big(t,x,\omega(x), \mathsmaller{\int} m \textnormal{d}\mu \Big)}{\R^d}{\mu(x)},
\end{equation*}
satisfies the hypotheses \textnormal{\textbf{(H5)}} of Theorem \ref{thm:PMP_General} and its Wasserstein gradient is given by
\begin{equation*}
\begin{aligned}
\nabla_{\mu} L(t,\mu,\omega)(x) = \nabla_x l \Big(t,x,\omega(x), \mathsmaller{\int} m \textnormal{d}\mu \Big) & + \D_x \omega(x)^{\top} \nabla_v l \Big(t,x,\omega(x), \mathsmaller{\int} m \textnormal{d}\mu \Big) \\
& + \D_x m(x)^{\top} \INTDom{\nabla_r l \Big(t,y,\omega(y), \mathsmaller{\int} m \textnormal{d}\mu \Big)}{\R^d}{\mu(y)}.
\end{aligned}
\end{equation*}
\end{prop}

\begin{prop}[Example of state constraints]
Let $m \in C^2(\R^d,\R^k)$ and $\lambda \in C^2([0,T] \times \R^d \times \R^k,\R^r)$. Then for any $l \in \{ 1,\dots,r\}$, the functionals
\begin{equation*}
\Lambda_l(t,\mu) \in [0,T] \times \Pcal(K) \mapsto \INTDom{\lambda_l \Big( t,x, \mathsmaller{\int} m \textnormal{d}\mu \Big)}{\R^d}{\mu(x)}
\end{equation*} 
satisfy the hypotheses \textnormal{\textbf{(H6)}} of Theorem \ref{thm:PMP_General} and their derivatives can be computed using Propositions \ref{prop:Examp_Cost}.
\end{prop}

\begin{rmk}
Particular cases of functionals which are of great interest for applications are for instance the variance functional $\mu \mapsto \tfrac{1}{2} \INTDom{|x-\bar{\mu}|^2}{\R^d}{\mu(x)}$ where $\bar{\mu} = \int y \, \textnormal{d} \mu(y)$ or the target-support map to a closed set $S \subset \R^d$ $\mu \mapsto \tfrac{1}{2}\INTDom{d_S(x)^2}{\R^d}{\mu(x)}$.
\end{rmk}


\section{Wasserstein differential of the running constraint penalization}
\label{appendix:ConstraintPenalization}

In this Section, we give the analytical expression of the Wasserstein differential of the running constraint penalization map $(t,\mu,\zeta,\omega) \mapsto
 \Ccal(t,\mu,\zeta,\omega)$ defined in \eqref{eq:Penalized_Constraints}.

\begin{prop}[Wasserstein differential of the state constraints penalization map]
\label{prop:WassersteinDiff_StateConstraints}
Let $K \subset \R^d$ be a compact set and $\omega \in U$. Under hypotheses \textnormal{\textbf{(H6)}}, the map
\begin{equation*}
\begin{aligned}
\mu \in \Pcal(K) \mapsto \Ccal(t,\mu,\zeta^*(t),\omega) & = \sum_{l=1}^r \zeta_l^*(t) \left( \partial_t \Lambda_l(t,\mu) + \INTDom{ \langle \nabla_{\mu} \Lambda_l(t,\mu)(x) , v[\mu](t,x) + \omega(x) \rangle}{\R^d}{\mu(x)} \right) \\
& = \sum_{l=1}^r \zeta_l^*(t) \Ccal_l(t,\mu,\omega)
\end{aligned}
\end{equation*}
is regular and Wasserstein-differentiable at at any $\mu \in \Pcal(K)$. The Wasserstein gradients of its summands $\Ccal_l(t,\cdot,\omega)$ can be computed explicitly as
\begin{equation}
\label{eq:WassersteinGradient_StateConstraint}
\begin{aligned}
\nabla_{\mu} \Ccal_l(t,\mu,\omega)(x) & = \partial_t \nabla_{\mu} \Lambda_l(t,\mu)(x) + \D_x \nabla_{\mu} \Lambda_l(t,\mu)(x)^{\top} \big( v[\mu](t,x) + \omega(x) \big) \\
& + \left( \D_x v[\mu](t,x)^{\top} + \D_x u^*(t,x)^{\top} \right) \nabla_{\mu} \Lambda_l(t,\mu)(x) + \INTDom{\BGamma^v_{(t,y)}(x)^{\top} \nabla_{\mu}\Lambda_l(t,\mu)(y)}{\R^d}{\mu(y)} \\
& + \INTDom{\BGamma^{\nabla_{\mu} \Lambda_l}_{(t,y)}(x)^{\top} \big(  v[\mu](t,y) + u^*(t,y) \big)}{\R^d}{\mu(y)}
\end{aligned}
\end{equation}
where $(t,x,y) \mapsto \BGamma^{\nabla_{\mu} \Lambda_l}_{(t,y)}(x) = \left(\BGamma^{\nabla_{\mu} \Lambda_l,i}_{(t,y)}(x) \right)_{1 \leq i \leq d}$ are the matrix-valued maps which rows are the Wasserstein gradients of the components of $\nabla_{\mu} \Lambda_l^i(t,\mu)(x)$, i.e.
\begin{equation*}
\BGamma^{\nabla_{\mu} \Lambda_l,i}_{(t,y)}(x) = \nabla_{\mu} \Big[ \nabla_{\mu} \Lambda_l^i(t,\cdot)(y) \Big](\mu)(x).
\end{equation*}
\end{prop}

\begin{proof}
In order to lighten the coming computations, we introduce the auxiliary functions
\begin{equation*}
\Ccal_l^1(t,\mu) = \partial_t \Lambda_l(t,\mu) \qquad \text{and} \qquad s\Ccal_l^2(t,\mu,\omega) = \INTDom{\langle \nabla_{\mu}\Lambda_l(t,\mu)(x) , v[\mu](t,x) + \omega(x) \rangle}{\R^d}{\mu(x)}.
\end{equation*}
Let $t \in [0,T]$ and $\mu \in \Pcal(K)$. The Wasserstein gradient of $\Ccal_l^1(t,\cdot)$ at $\mu$ is given by
\begin{equation}
\label{eq:StateConstraint_Derivatives1}
\nabla_{\mu} \Ccal_l^1(t,\mu) = \nabla{\mu} \left( \partial_t \Lambda_l(t,\cdot) \right)(\mu) = \partial_t \nabla_{\mu} \Lambda_l(t,\mu)
\end{equation}
We turn our attention to $\Ccal_l^2(t,\cdot,\omega)$. For any $\nu \in \Pcal(K)$ and $\gamma \in \Gamma_o(\mu,\nu)$, it holds that
\begin{equation}
\label{eq:StateConstraint_Derivatives2}
\begin{aligned}
& \hspace{0.4cm} \Ccal_l^2(t,\nu,\omega) - \Ccal^2_l(t,\mu,\omega) \\
& = \INTDom{\Big( \langle \nabla_{\mu} \Lambda_l(t,\nu)(y) , v[\nu](t,y) + u^*(t,y) \rangle -  \langle \nabla_{\mu} \Lambda_l(t,\mu)(x) , v[\mu](t,x) + \omega(x) \rangle \Big)}{\R^{2d}}{\gamma(x,y)} \\
& = \INTDom{\Big( \langle \nabla_{\mu} \Lambda_l(t,\nu)(x) , v[\nu](t,x) + \omega(x) \rangle - \langle \nabla_{\mu} \Lambda_l(t,\mu)(x) , v[\mu](t,x) + \omega(x) \rangle\Big)}{\R^d}{\mu(x)} \\
& + \INTDom{\left\langle \D_x \nabla_{\mu} \Lambda_l(t,\mu)(x)^{\top} \left( v[\mu](x) + \omega(x) \right) , y - x \right\rangle}{\R^{2d}}{\gamma(x,y)} \\
& + \INTDom{\left\langle \left( \D_x v[\mu](t,x)^{\top} + \D_x \omega(x)^{\top} \right) \nabla_{\mu} \Lambda_l(t,\mu)(x) , y - x \right\rangle}{\R^{2d}}{\gamma(x,y)} + \INTDom{o(|x-y|)}{\R^{2d}}{\gamma(x,y)} 
\end{aligned}
\end{equation}
By definition of the Landau notation $o(\cdot)$, for any $\epsilon > 0$ there exists $\eta > 0$ such that whenever $|x-y| \leq \eta$, one has that $o(|x-y|) \leq \epsilon |x-y|$. Therefore,
\begin{equation*}
\begin{aligned}
\INTDom{o(|x-y|)}{\R^{2d}}{\gamma(x,y)} & \leq \epsilon \INTDom{|x-y|}{\{ |x-y| \leq \eta\}}{\gamma(x,y)} + C \INTDom{|x-y|}{\{ |x-y| > \eta\}}{\gamma(x,y)} \\
& \leq \epsilon W_2(\mu,\nu) + 2C \text{diam}(K) \, \gamma \left( \big\{ (x,y) \in \R^{2d} ~\text{s.t.}~  |x-y| > \eta \big\} \right) \\
& \leq \epsilon W_2(\mu,\nu) + \frac{2C}{\eta^2} \text{diam}(K) W_2^2(\mu,\nu)
\end{aligned}
\end{equation*}
by Chebyshev's inequality and where the constant $C > 0$ exists because $o(|x-y|)$ is in particular a $O(|x-y|)$ on the compact set $\supp(\gamma) \subset \R^{2d}$ for $|x-y| > \eta$. Upon choosing $\eta' = \eta^2 \epsilon/(2C \text{diam}(K))$, we recover that
\begin{equation*}
\INTDom{o(|x-y|)}{\R^{2d}}{\gamma(x,y)} \leq 2 \epsilon W_2(\mu,\nu).
\end{equation*}
whenever $W_2(\mu,\nu) \leq \eta'$. By definition, this estimate precisely amounts to the fact that $\INTDom{o(|x-y|)}{\R^{2d}}{\gamma(x,y)} = o(W_2(\mu,\nu))$ as $W_2(\mu,\nu) \rightarrow 0$.

We further compute the first-order variations arising from the remaining measure terms as follows  
\begin{equation}
\label{eq:StateConstraint_Derivatives3}
\begin{aligned}
& \INTDom{\Big( \langle \nabla_{\mu} \Lambda_l(t,\nu)(x) , v[\nu](t,x) + \omega(x) \rangle - \langle \nabla_{\mu} \Lambda_l(t,\mu)(x) , v[\mu](t,x) + \omega(x) \rangle\Big)}{\R^d}{\mu(x)} \\
= ~ & \INTDom{\left\langle \INTDom{\BGamma^{\nabla_{\mu}\Lambda_l}_{(t,x)}(x')(y'-x')}{\R^{2d}}{\gamma'(x',y')} , v[\mu](t,x) + \omega(x) \right\rangle}{\R^d}{\mu(x)} \\
+ ~ & \INTDom{\left\langle \INTDom{\BGamma^v_{(t,x)}(x')(y'-x')}{\R^{2d}}{\gamma'(x',y')} , \nabla_{\mu} \Lambda_l(t,\mu)(x) \right\rangle}{\R^d}{\mu(x)}   + o(W_2(\mu,\nu)) \\
= ~ & \INTDom{\left\langle \INTDom{\Big( \BGamma^v_{(t,y')}(x)^{\top} \nabla_{\mu}\Lambda_l(t,\mu)(y') + \BGamma^{\nabla_{\mu}\Lambda_l}_{(t,y')}(x)^{\top} \nabla_{\mu}\Lambda_l(t,\mu)(y') \Big)}{\R^d}{\mu(y')} , y - x \right\rangle}{\R^{2d}}{\gamma(x,y)} + o(W_2(\mu,\nu))
\end{aligned}
\end{equation}
by a standard application of Fubini's Theorem. Merging equations \eqref{eq:StateConstraint_Derivatives1}, \eqref{eq:StateConstraint_Derivatives2} and \eqref{eq:StateConstraint_Derivatives3}, we recover the characterization \eqref{eq:WassersteinDiff} of the Wasserstein gradient $\nabla_{\mu} \Ccal_l(t,\mu,\omega)(\cdot)$ of $\Ccal_l(t,\cdot,\omega)$ at $\mu$ given by \eqref{eq:WassersteinGradient_StateConstraint}.
\end{proof}


\bibliographystyle{plain}
{\footnotesize
\bibliography{../../ControlWassersteinBib}}

\end{document}